\documentclass[a4paper,12pt]{amsart}
\usepackage{amscd,amsmath,amssymb}
\usepackage{latexsym,amsbsy,mathrsfs}
\usepackage{xy}
\usepackage{xypic}
\usepackage{pgf,tikz,bm}

\newtheorem{theorem}{Theorem}
\newtheorem{conjecture}[theorem]{Conjecture}
\newtheorem{thm}{Theorem}[section]
\newtheorem{lem}[thm]{Lemma}
\newtheorem{cor}[thm]{Corollary}
\newtheorem{prop}[thm]{Proposition}

\numberwithin{equation}{section}
\usepackage{amssymb,amsthm,array}

\usepackage{lmodern}
\usepackage{verbatim}

\theoremstyle{definition}
\newtheorem{example}[thm]{Example}

\newtheorem{rem}[thm]{Remark}

\newcommand{\rankv}{\operatorname{\underline{rank}}}

\newcommand{\rank}{\operatorname{rank}}
\newcommand{\End}{\operatorname{End}}
\newcommand{\type}[1]{\operatorname{\widetilde{\mathbb{#1}}}}
\newcommand{\rep}{\operatorname{rep}}
\newcommand{\lfrep}{\operatorname{rep}_{\mathrm{lf}}}
\newcommand{\delE}{H'}

\newcommand{\val}{Q_{\mathrm{val}}(C)}
\renewcommand{\dim}{\operatorname{dim}}
\newcommand{\dimv}{\operatorname{\underline{dim}}}

\usepackage[linktocpage,pdfstartview=FitH,colorlinks=true,
linkcolor=blue,citecolor=magenta]{hyperref}
\usepackage{amsthm}

\normalbaselines
\title[On affine type Geiss-Leclerc-Schr\"{o}er's conjecture]
{Affine root systems, stable tubes and a conjecture by Geiss-Leclerc-Schr\"{o}er}

\author{Zengqiang Lin}
\address{ Zengqiang Lin: School of Mathematical sciences, Huaqiao University,
Quanzhou\quad 362021,  China.}
\email{lzq134@163.com}
\author{Xiuping Su}
\address{Xiuping Su: Department of Mathematical Sciences, University of Bath, Bath BA2 7JY, UK.}
\email{xs214@bath.ac.uk}

\thanks{The first author was partially supported by  the National Natural Science Foundation of China (Grant No. 12371037).}

\begin{document}
\maketitle

\begin{abstract}
Associated to a  symmetrisable Cartan matrix $C$, Geiss-Lerclerc-Schr\"{o}er  constructed and studied a
class of Iwanaga-Gorenstein algebras
 $H$. They proved a generalised version of Gabriel's Theorem,  that is, the rank vectors of $\tau$-locally free $H$-modules are
the  positive roots of type $C$ when $C$ is of finite type, and conjectured that this is true for any $C$.
 In this paper, we look into this conjecture when $C$ is of affine type.
We construct explicitly stable tubes, some of which have rigid mouth modules, 
while others not. 
We deduce that  any positive root of type $C$ is the rank vector of
some $\tau$-locally free $H$-module. However, the converse is not true in general. Our construction shows  that
there are  $\tau$-locally free $H$-modules whose  rank vectors are not roots,
when $C$ is of type $\type{B}_n$, $\type{CD}_n$, $\type{F}_{41}$ and $\type{G}_{21}$,
and so the conjecture fails in these four types.
\end{abstract}

\setcounter{tocdepth}{1}
\tableofcontents

\section{Introduction}
 Gabriel's Theorem \cite{[Gab]}, a fundamental result in quiver representation theory, displays an
 important connection between the isomorphism classes of indecomposable representations of  connected Dynkin quivers and the positive roots of the corresponding Dynkin diagrams (or  complex simple Lie algebras). That is,
dimension vectors of indecomposable representations of a Dynkin quiver $Q$
are the positive roots of the underlying graph of $Q$, with simple representations corresponding to simple roots.
The result was later generalised to affine quivers by Nazarova \cite{[Naz]}, to
Dynkin or affine valued quivers by Dlab-Ringel \cite{[DR]} and to arbitrary quivers by Kac \cite{[K1]}.
Recently, Geiss-Lecerc-Schr\"{o}er  \cite{[GLS1]} extended the correspondence further
to a class of Iwanaga-Gorenstein algebras
$H=H(C, D, \Omega)$ associated to a symmetrisable Cartan matrix $C$ with symmetriser $D$
and any orientation $\Omega$ that has no oriented cycles.

One distinguished feature in Geiss-Lerclerc-Schr\"{o}er's work is that
they did not consider all  indecomposable $H$-modules, but only the so called $\tau$-locally free $H$-modules and
their rank vectors instead of dimension vectors.
When $C$ is of Dynkin type, they \cite[Theorem 1.3]{[GLS1]} showed that  the map $M\mapsto \rankv M$ gives a bijection between the isomorphism classes of $\tau$-locally free $H$-modules and the positive roots of the complex simple Lie algebra of type $C$.
In particular, when $C$ is symmetric and the symmetriser $D$ is minimal, they recovered
Gabriel's Theorem \cite{[Gab]}. They further proposed the following conjecture for the general case.

\vspace{2mm}\noindent
\begin{conjecture}\cite[Conjecture 5.3]{[GLS2]} \label{Conj:gls}
Let $H=H(C,D,\Omega)$ with the Cartan matrix $C$ of size $n\times n$ and let
$\alpha\in \mathbb{Z}_{\geq 0}^n$.
Then $\alpha$ is a root of the Kac-Moody Lie algebra of type $C$ if and only if $\alpha$ is   the rank vector of some
$\tau$-locally free $H$-module.
\end{conjecture}

\vspace{2mm}

For any $C$, Geiss-Leclerc-Schr\"{o}er \cite{[GLS6]}  proved that there is a bijection between  the rank vectors of rigid $\tau$-locally free  $H$-modules and  the real Schur roots. Recently,
Huang-Lin-Su \cite{[HLS]} showed that Conjecture 1 is true when $C$ is of type $\type{C}_n$ and the symmetriser $D$ is minimal.
In this paper, we  continue the investigation of the connection between representations and root systems displayed in Conjecture \ref{Conj:gls} and achieve  the following.

\vspace{2mm}
\begin{theorem}[Theorem \ref{thm3.1}] \label{thmA}
 Let $C$ be a symmetrisable Cartan matrix of affine type and $H=H(C,D,\Omega)$.  Then for any positive root $\alpha$
of type $C$, there exists a $\tau$-locally free $H$-module $M$ such that $\rankv M=\alpha$.
\end{theorem}

\vspace{2mm}

 \begin{theorem} [Theorem \ref{main2}] \label{thmB}
Let $C$ be a symmetrisable Cartan matrix of type $\type{B}_n$, $\type{CD}_n$, $\type{F}_{41}$ or $\type{G}_{21}$, $D$  a minimal symmetriser and $H=H(C,D,\Omega)$.
\begin{itemize}
\item[(1)] The AR-quiver $\Gamma_H$ has an inhomogeneous tube of $\tau$-locally free modules,
whose mouth modules are not rigid and have the minimal positive imaginary root as their rank vectors.

\item[(2)] There exist  $\tau$-locally free $H$-modules such that  their rank vectors
 are not roots. Consequently, Conjecture 1 fails in these four types.
 \end{itemize}
 \end{theorem}

\vspace{2mm}

The remaining part of this paper is organised as follows. In \S  \ref{sec2}, we recall some basic facts on the Iwanaga-Gorenstein algebras $H=H(C,D,\Omega)$, $\tau$-locally free $H$-modules and roots.
In \S \ref{sec3}, we
 show that Conjecture \ref{Conj:gls} is independent of orientations and
start the proof of  Theorem \ref{thmA}. We complete the proof
in \S \ref{Sec:homogtubes} and \S \ref{sec4}.
In particular, in \S \ref{Sec:homogtubes}, we construct homogeneous tubes.
In  \S \ref{sec4}, we construct inhomogeneous tubes with rigid mouth modules. Consequently, the rank vectors of the mouth modules are real Schur roots. In \S \ref{sec5}, we prove Theorem \ref{thmB}. We construct  
inhomogeneous tubes with non-rigid mouth modules, when $C$ is of type  $\type{B}_n$, $\type{CD}_n$, $\type{F}_{41}$ or $\type{G}_{21}$, and $D$ is minimal.
In particular, we find $\tau$-locally free modules, whose rank vectors are not roots.

\section{The Iwanaga-Gorenstein algebras $H$ and $\tau$-locally free $H$-modules}
\label{sec2}
We recall some basic definitions and facts on the Iwanaga-Gorenstein algebras $H=H(C, D, \Omega)$ constructed by Geiss-Leclerc-Schr\"{o}er, $\tau$-locally free $H$-modules  and their rank vectors. We refer the reader
to \cite{[GLS1]} for more details.

\subsection{The algebra $H(C,D,\Omega)$}
Let $C=(c_{ij})\in M_n(\mathbb{Z})$ be  a {\em symmetrisable generalised Cartan matrix}, which we often simply refer
to as a {\em Cartan matrix}, and let $D=\text{diag}(d_1,...,d_n)$ be a symmetriser of $C$.
%
 When $c_{ij}<0$, define \[g_{ij}=|\text{gcd}(c_{ij},c_{ji})| \text{ and } f_{ij}=|c_{ij}|/g_{ij}.\]
An {\em orientation} of $C$ is a subset $\Omega\subset\{1,2,\cdots,n\}\times\{1,2,\cdots,n\}$ such that
\begin{itemize}
\item[(O1)] $\{(i,j),(j,i)\}\cap\Omega\neq\emptyset$ if and only if $c_{ij}<0$;
\item[(O2)] If $((i_1,i_2),(i_2,i_3),\cdots,(i_t,i_{t+1}))$ is a sequence with $t\geq 1$ and $(i_s,i_{s+1})\in\Omega$ for all $1\leq s\leq t$, then $i_1\neq i_{t+1}$.
\end{itemize}

For an orientation $\Omega$ of $C$, let $Q=Q(C,\Omega)=(Q_0, Q_1)$ be the  quiver  with  vertices \[Q_0=\{1,2,\cdots,n\}\] 
and arrows
$$Q_1=\{\alpha_{ij}^{(g)}\colon j\rightarrow i|(i,j)\in\Omega,1\leq g\leq g_{ij}\}\cup
\{\varepsilon_i \colon i\rightarrow i|1\leq i\leq n\}.$$
 If $g_{ij}=1$, we write $\alpha_{ij}$ instead of $\alpha_{ij}^{(1)}$. Let $Q^0=Q^0(C,\Omega)$ be the quiver obtained from $Q$ by deleting all loops $\varepsilon_i$.
(O2) implies that $Q^0$ is an acyclic quiver.

\vspace{1mm}
Let $H=H(C,D,\Omega)=KQ/I$, where $KQ$ is the path algebra of $Q$ over a field $K$,
and $I$ is the ideal of $KQ$ defined by the following relations:
\begin{itemize}
\item[(H1)] For each $i\in Q_0$ we have $\varepsilon_i^{d_i}=0$;

\item[(H2)] For each $(i,j)\in\Omega$ and each $1\leq g\leq g_{ij}$ we have $\varepsilon_i^{f_{ji}}\alpha_{ij}^{(g)}=\alpha_{ij}^{(g)}\varepsilon_j^{f_{ij}}$.
\end{itemize}
\noindent Let $H_i$ be the subalgebra of $H$ generated by $\varepsilon_i$. Then $H_i$ is isomorphic to the truncated polynomial algebra $K[x]/(x^{d_i})$.

\subsection{$\tau$-locally free modules}
Let $M=(M_i,M(\varepsilon_i),M(\alpha^{(g)}_{ij}))_{i\in Q_0,(i,j)\in\Omega}$ be a representation of $H$.
That is, $M_i$ is a finite-dimensional $K$-vector space, $M(\varepsilon_i):M_i\rightarrow M_i$ and $M(\alpha^{(g)}_{ij}): M_j\rightarrow M_i$ are $K$-linear maps satisfying the relations (H1) and (H2).
(H1) implies that each $M_i$ is an $H_i$-module.
A (left) $H$-module $M$ is said to be {\em locally free} if $M_i$ is a free $H_i$-module for each $i$, and in this case
we define the rank vector of $M$ as\[\rankv M=(\rank M_1, \dots, \rank M_n), \]
where $\rank M_i$ is the rank of $M_i$ as an $H_i$-module.
Denote by $\rep(H)$ the category of (finite dimensional) $H$-modules and by $\lfrep(H)$ the full subcategory of locally free
$H$-modules.
Recall that the homological bilinear form of  locally free $H$-modules,
$$\langle M,N\rangle= \text{dim Hom}_H(M,N)-\text{dim Ext}^1_H(M,N).$$

\begin{prop} \cite[Proposition 4.1]{[GLS1]} \label{lem2.2}
Let $M$ and $N$ be locally free $H$-modules. Then
\begin{equation}\label{eq:bil}\langle M,N\rangle= \Sigma_{i=1}^n d_ia_ib_i-\Sigma_{(j,i)\in\Omega}d_i |c_{ij}| a_ib_j,
\end{equation}
where $\rankv M=(a_1,\cdots, a_n)$ and $\rankv N=(b_1,\cdots, b_n)$.
\end{prop}

An indecomposable
$H$-module $M$ is called {\em $\tau$-locally free}, if $\tau_H^k(M)$ is locally free for all $k\in\mathbb{Z}$, where $\tau_H$ is the Auslander-Reiten translation. 

\begin{prop} \cite[Proposition 11.4]{[GLS1]} \label{lem2.1}
Let $M$ be an indecomposable locally free and rigid $H$-module. Then $M$ is $\tau$-locally free and $\tau^k(M)$ is rigid for all $k\in\mathbb{Z}$.
\end{prop}

Let $i\in Q_0$.  Denote by $P_{i}$ (resp. $I_{i})$ the indecomposable projective (resp. injective) $H$-module associated with $i$, and by $E_{i}$  the regular representation $H_{i}$,
which is called the {\em generalised simple} $H$-module at vertex $i$. By Proposition \ref{lem2.1},
we know that $P_i$ and $I_i$ are $\tau$-locally free, since they are locally free and rigid. The next example shows that all the generalised simple
modules are also $\tau$-locally free.

\begin{example}\label{ex:gensimple}
Observe that $\dim \End E_i=d_i$ and $\langle E_i,E_i\rangle=d_i$. Therefore, by Proposition \ref{lem2.2},
$E_i$ is rigid and thus $\tau$-locally free by Proposition \ref{lem2.1}.
\end{example}

\subsection{The root system $\Delta$ of type $C$}
Let $\alpha_1,\alpha_2,\cdots,\alpha_n$ be a list of  positive simple roots of the Kac-Moody algebra $g(C)$ associates with $C$. We identify {\em $\alpha_i$ with the rank vector of the generalised simple $H$-module
$E_i$ at vertex $i$}, and
also say that positive roots of $g(C)$ with respect to the fixed simple roots are {\em positive roots of type} $C$.

For any $1\leq i\leq  n$, the simple reflection associated to $i$ is defined by
$$s_i:\mathbb{Z}^{n}\longrightarrow \mathbb{Z}^{n}, \; \alpha_j\mapsto \alpha_j-c_{ij}\alpha_i,$$
where $\mathbb{Z}^{n}=\sum_{i=1}^{n}\mathbb{Z}\alpha_i$ is the root lattice of $g(C)$.
The {\em Weyl group} $W$ is the subgroup of $\textup{Aut}(\mathbb{Z}^{n})$ generated by $s_1,s_2,\cdots,s_n$.
Denote by $\Delta_{\textup{re}}=\cup_{i=1}^{n}W(\alpha_i)$  the set of {\em real roots}, and by
 $\Delta_{\textup{im}}$  the set of {\em imaginary roots}. When
 $C$ is of affine type, \[\Delta_{\textup{im}}=\mathbb{Z}\delta,\]
 where $\delta$ is the unique minimal positive imaginary root determined by the Cartan matrix
 $C$.
The {\em root system of type} $C$ is
 \[\Delta=\Delta_{\textup{re}}\cup\Delta_{\textup{im}}.\]

For an orientation $\Omega$ of $C$ and an {\em admissible vertex} $i$ in $Q^0(C,\Omega)$,
that is, $i$ is either a sink or a source vertex, let
$$s_i(\Omega)=\{(r,s)\in\Omega\mid i\notin\{r,s\}\}\cup\{(s,r) \mid i\in\{r,s\},(r,s)\in\Omega\}.$$
Then $s_i(\Omega)$ is again an orientation of $C$.
A sequence $\mathbf{i}=(i_1,i_2,\cdots,i_n)$ is called a {\em $+$-admissible sequence} for $(C,\Omega)$ if $\{i_1,i_2,\cdots,i_n\}=\{1,2,\cdots,n\}$, $i_1$ is a sink in $Q^0(C,\Omega)$ and $i_k$ is a sink in $Q^0(C,s_{i_{k-1}}\cdots s_{i_1}(\Omega))$ for $2\leq k\leq n$.

Now we fix a $+$-admissible sequence $\mathbf{i}=(i_1,i_2,\cdots,i_n)$.
Define $$\beta_{\mathbf{i},k}=\beta_k=\left\{\begin{array}{ll}
                \alpha_{i_1}, & \text{if}\ k=1; \\
                s_{i_1}s_{i_2}\cdots s_{i_{k-1}}(\alpha_{i_{k}}), & \text{if}\ 2\leq k\leq n.
              \end{array}
\right.$$
Similarly, define
$$\gamma_{\mathbf{i},k}=\gamma_k=\left\{\begin{array}{ll}  \alpha_{i_n}, & \text{if}\ k=n; \\
                s_{i_n}s_{i_{n-1}}\cdots s_{i_{k+1}}(\alpha_{i_k}), & \text{if}\ 1\leq k\leq n-1. \\
                      \end{array}
\right.$$

\begin{lem} \cite[Lemma 3.2, Lemma 3.3]{[GLS1]} \label{lem:rkvpreprojinj}
The following are true.
\[\rankv E_{i_k}=\alpha_{i_k}, ~\rankv P_{i_k}=\beta_k \text{ and }
\rankv I_{i_k}=\gamma_k.\]
\end{lem}

\subsection{Coxeter transformations, reflection functors and AR-translation}
Assume that $\mathbf{i}=(i_1,i_2,\cdots,i_n)$ is a $+$-admissible sequence  for $(C,~\Omega)$. Let
 \[c_{\mathbf{i}}=s_{i_n}s_{i_{n-1}}\cdots s_{i_1}:\mathbb{Z}^{n}\rightarrow\mathbb{Z}^{n}.\]
Then $c_\mathbf{i}^{-1}=s_{i_1}s_{i_{2}}\cdots s_{i_n}$.
We call $c_\mathbf{i}$ and $c_\mathbf{i}^{-1}$ {\em Coxeter transformations}  associated to $\mathbf{i}$.
If there is no confusion,  we denote $c_\mathbf{i}$ simply by $c$.

Note that the rotated sequence
\[\mathbf{i}'=(i_2,i_3,\cdots,i_n, i_1)\]
is also a $+$-admissible sequence for $(C,s_{i_1}(\Omega))$, and $c_{\mathbf{i'}}=s_{i_1}s_{i_n}\cdots s_{i_{3}} s_{i_2}$ and $c_{\mathbf{i'}}^{-1}=s_{i_2}s_{i_3}\cdots s_{i_{n}} s_{i_1}$ are the Coxeter transformations associated to $\mathbf{i'}$.

Let $i$ be a sink (resp. source) vertex in $Q^0$,  let
$s_{i}(H)=H(C,D,s_{i}(\Omega))$ and denote by
\[F^+_{i} (\text{resp. } F^-_i): \text{rep}(H)\rightarrow \text{rep}(s_{i}(H))\]
the {\em reflection functor} at  $i$.

\begin{prop} (see \cite[Proposition 9.4, Proposition 11.8]{[GLS1]}) \label{lem6}
Let $i$ be a sink (resp. source) vertex and
$M$ a $\tau$-locally free  $H$-module such that  $M\not\cong E_i$.
Then $F^+_{i}(M)$ (resp. $F^-_{i}(M)$) is a $\tau$-locally free $s_{i}(H)$-module.
Moreover,
\[ F^{-}_{i}F^+_{i}(M)\cong M  (\text{ resp. }\ F^{+}_{i}F^-_{i}(M)\cong M)\] and
$\rankv F^{+}_{i}(M)=s_{i}(\rankv M)$ (resp. $\rankv F^{-}_{i}(M)=s_{i}(\rankv M)$).
\end{prop}

Given a $+$-admissible sequence $\mathbf{i}=(i_1,i_2,\cdots,i_n)$  for $(C,~\Omega)$, let 
$$C^+=F^+_{i_n}\cdots F^+_{i_2}F^+_{i_1}:\text{rep}(H)\rightarrow \text{rep}(H).$$ 
Dually, one defines 
$$C^-=F^-_{j_n}\cdots F^-_{j_2}F^-_{j_1}:\text{rep}(H)\rightarrow \text{rep}(H)$$
for a $-$-admissible sequence  $\mathbf{j}=(j_1,j_2,\cdots,j_n)$.
 We call $C^+$ and $C^-$ {\em Coxeter functors}. Let  $T:\text{rep}(H)\longrightarrow \text{rep}(H)$ be the twist automorphism induced from the algebra automorphism $H\rightarrow H$ defined by $\varepsilon_i \mapsto \varepsilon_i$ and $\alpha_{ij}^{(g)}\mapsto -\alpha_{ij}^{(g)}$.

\begin{prop} \cite[Proposition 11.9]{[GLS1]} \label{lem:TC}
Let $M$ be an $H$-module. The following are equivalent.
\begin{itemize}
\item[(1)] $M$ is locally free.

\item[(2)] $\tau (M)\cong TC^+(M)$.

\item[(3)] $\tau^-(M) \cong TC^-(M)$.
\end{itemize}
\end{prop}

\begin{prop} \cite[Proposition 11.5]{[GLS1]}\label{Lem:tauc}
Let $M\in \textup{rep}(H)$ be a $\tau$-locally free module. If $\tau^k(M)\neq 0$, where $k\in \mathbb{Z}$, then
$\rankv \tau^k(M)=c^k(\rankv M)$.
\end{prop}

\begin{prop} \cite[Proposition 11.6]{[GLS1]}\label{lem4}
Let $M$ be an indecomposable preprojective or preinjective $H$-module.
\begin{itemize}
\item[(1)] $M$ is $\tau$-locally free and rigid.

\item[(2)] $\rankv M \in \Delta^+_{\textup{re}}$.

\item[(3)] If  $N$ is also an indecomposable preprojective or preinjective $H$-module with
$\rankv M=\rankv N,$ then $M\cong N$.
\end{itemize}
\end{prop}

\begin{thm}\cite[Theorem 1.2 (a)]{[GLS6]}\label{thm:Schurroot}
The map $M\mapsto \rankv M$ induces a one-to-one correspondence between
the isomorphism classes of rigid $\tau$-locally free $H$-modules and real Schur roots associated to
the admissible sequence $\mathbf{i}$.
\end{thm}

\section{Rank vectors of  $\tau$-locally free $H$-modules and positive roots}
\label{sec3}
\subsection{Regular components of AR-quivers}
Let $A$ be a finite dimensional $K$-algebra and let $\Gamma_A$ be
the Auslander-Reiten quiver of $A$.
An indecomposable  $A$-module $M$  is called {\em regular} if $\tau^k(M)\neq 0$ for all $k\in\mathbb{Z}$. A connected component of $\Gamma_A$ is called {\em regular} if it only consists of regular modules.
We recall two facts on regular components of AR-quivers.

\begin{thm}(see \cite[Theorem XX.1.2]{[SS]}) \label{lem1}
Let $A$ be a finite dimensional $K$-algebra and $\mathcal{C}$ a regular component of $\Gamma_A$.
\begin{itemize}
\item[(1)]
$\mathcal{C}$ contains an oriented cycle if and only if it is a stable tube.

\item[(2)] $\mathcal{C}$ is acyclic if and only if it is of the form $\mathbb{Z}\Delta$ for some locally finite acyclic quiver
$\Delta$.
\end{itemize}
\end{thm}

We remark that Theorem \ref{lem1} is Simson-Skowro\'{n}ski's summary of some results proved independently
by Liu \cite{[Liu1]} and Zhang \cite{[Zhang]}.
Tubes to be considered in this paper are all stable. So later on we will just say a tube instead of a stable tube.  

\begin{prop}\cite[Corollary 3.13]{[Liu1]}\label{Liu1} Let $\mathcal{C}$ be a regular component of $\Gamma_A$.
Assume that there exist a module $X\in \mathcal{C}$ and a constant $x$ such that there are infinitely
many $r$ with  $\dim\tau^{r}X$ bounded by $x$. Then $\mathcal{C}$ is either a tube or of the form
$\mathbb{Z}\mathbb{A}_{\infty}$.
\end{prop}

\subsection{ More on $\tau$-locally free  modules}
In the remaining of this paper,
we assume $C$ is of affine type, and we fix
a  $+$-admissible sequence $\mathbf{i}=(i_1,i_2,\cdots,i_n)$.

\begin{prop}\label{lem3}
Let $M$ be a $\tau$-locally free regular $H$-module. Then there exists some positive integer $N$ such that
$c^N(\rankv M)=\rankv M$.
\end{prop}
Recall from Proposition \ref{Lem:tauc} that
$\rankv \tau^k(M)=c^k(\rankv M)$. The proof of  Lemma 7.1 in \cite{[CB]} can be
 applied to prove Proposition \ref{lem3}.
 The key fact is that there exists some integer $N>0$ and $r$
such that
\begin{equation}\label{eq:c} c^{iN} d=d + ir\delta \end{equation}
for all $d\in \mathbb{Z}^{n}$ and $i\in \mathbb{Z}$.
We refer the reader to \cite{[CB]} for the details.

We say that a vector $\alpha$ is {\em $c$-periodic with $c$-period $r$} if $r$ is the minimal positive integer such that $c^r\alpha=\alpha$.
Similarly, we say that a  $\tau$-locally free $H$-module $M$ is {\em $\tau$-periodic with $\tau$-period $r$} if $r$ is the minimal positive integer such that $\tau^rM\cong M$.

\begin{prop} \cite[Proposition 11.14]{[GLS1]}\label{lem5}
Let $\mathcal{C}$ be a connected component of $\Gamma_H$, then
$\mathcal{C}$ contains a $\tau$-locally free regular module if and only if $\mathcal{C}$ only consists
of  $\tau$-locally free regular modules.
\end{prop}

\begin{prop}\label{cor}
Let $\mathcal{C}$ be a connected component of $\Gamma_H$ that contains a regular $\tau$-locally free $H$-module.
Then $\mathcal{C}$ is either a tube or of the form
$\mathbb{Z}\mathbb{A}_{\infty}$. Furthermore, if  $\mathcal{C}$  contains a $\tau$-periodic module, then
 $\mathcal{C}$  is a tube.
\end{prop}

\begin{proof} Let $M$ be a regular  $\tau$-locally free module contained in  $\mathcal{C}$.
By Proposition \ref{lem3},  $\rankv M$  is $c$-periodic. Consequently,  $\text{dim} \tau^rM$ for all $r$
are bounded by a constant. By Proposition \ref{lem5}, we know that $\mathcal{C}$ is a regular component and
so by Proposition \ref{Liu1}, the component $\mathcal{C}$ is as claimed.

If $\mathcal{C}$ in addition  contains a $\tau$-periodic module, by Theorem \ref{lem1},  $\mathcal{C}$  is a tube.
\end{proof}

\begin{prop}\label{rem1}
Let $M$ be a $\tau$-locally free $H$-module.
\begin{itemize}
\item[(1)]  $M$ is preprojective if and only if $c^i(\rankv M)< 0$, when $i$ is large enough.

\item[(2)] $M$ is preinjective if and only if $c^{-i}(\rankv M)< 0$, when $i$ is large enough.

\item[(3)] $M$ is regular if and only if $\rankv M$ is $c$-periodic.
\end{itemize}
\end{prop}

\begin{proof}
Suppose that $M$ is preprojective, that is, $M=\tau^{-r}(P_{i_k})$ for some $r\geq 0$.
Then by Proposition \ref{Lem:tauc} and  straightforward computation,
 \[c^{r+1}(\rankv M)=c(\beta_k)=-\gamma_k<0,\]
 and so
 \[c^{s}(\rankv M)=-\rankv \tau^{s-r-1}I_{i_k}<0\]
 for any $s> r$.
Similarly, when $M$ is preinjective, then
$c^{-i}(\rankv M)< 0$, when $i$ is large enough.
When $M$ is regular, the $c$-periodicity of $\rankv M$ follows from Lemma \ref{lem3}.
Finally, the converses are self-evident, as the rank vectors of preprojective, preinjective
and regular modules behave differently under the Coxeter transformation.
\end{proof}

As a consequence of (\ref{eq:c}) and Proposition \ref{rem1}, we have the following.
\begin{cor} Let $j$ and $v$ be any vertices of $Q$.
\[\lim_{i\to \infty} (\rankv \tau^{-i}P_j)_v=\infty \text{ and }
\lim_{i\to \infty} (\rankv \tau^{i}I_j)_v=\infty.
\]
\end{cor}

\subsection{Conjecture \ref{Conj:gls} and the orientation $\Omega$}

Let $M$ be an indecomposable module in a  tube $\mathcal{C}$.
Recall that a {\em mouth module} is a module in a tube with the middle term
in the AR-sequence starting from (or ending at)
that module  indecomposable.

Recall that $\mathbf{i}=(i_1,i_2,\cdots,i_n)$ is a $+$-admissible sequence  for $(C,~\Omega)$. Let $H'=s_{i_1}(H)=H(C,D,s_{i_1}(\Omega))$.
 We denote by $P'_{i_k}$ (resp. $I'_{i_k}$) the indecomposable projective (resp. injective)  $H'$-module associated to $i_k$.

\begin{prop}\label{lem 2.1}
Let $M$ be a $\tau$-locally free $H$-module.

\begin{itemize}
\item[(1)] If $M=\tau_H^{-r}(P_{i_k})$ is preprojective and $M\not\cong P_{i_1}$, which is also $E_{i_1}$, then $F^+_{i_1}(M)$ is also preprojective and $$F^+_{i_1}(M)\cong\left\{ \begin{array}{ll}
                                                           \tau^{-r}_{H'}P'_{i_{k}}, & \text{ if }~ 2\leq k\leq n, r\geq0; \\ \\
                                                           \tau^{-r+1}_{H'}P'_{i_k}, &  \text{ if }~  k=1,r>0.
                                                         \end{array}\right.
$$

\item[(2)] If $M=\tau_H^{s}(I_{i_k})$ is preinjective, then $F^+_{i_1}(M)$ is also preinjective and  $$F^+_{i_1}(M)\cong \left\{ \begin{array}{ll}
                                                           \tau_{H'}^{s}I'_{i_{k}}, & \text{ if }~ 2\leq k\leq n, s\geq0; \\\\
                                                           \tau_{H'}^{s+1}I'_{i_k}, & \text{ if }~ k=1,s\geq0.
                                                         \end{array}\right.
$$

\item[(3)] If $M$ is regular, then $F^+_{i_1}(M)$ is  $\tau$-locally free regular.

\item[(4)] If $M$ is contained in  a  tube of rank $r$, then so is $F^+_{i_1}(M)$. Moreover,
if $M$ is a mouth module, then so is  $F^+_{i_1}(M)$.
\end{itemize}
\end{prop}

\begin{proof}
Note that $s_{i_1}c_{\mathbf{i}'}=c_{\mathbf{i}}s_{i_1}$, $s_{i_1}c^{-1}_{\mathbf{i}}=c^{-1}_{\mathbf{i}'}s_{i_1}$ and $\beta_{\mathbf{i},1}=\gamma_{\mathbf{i}',n}=\alpha_{i_1}$ by definition.

(1) If $M=\tau_H^{-r}(P_{i_k})$ is preprojective and $M\not\cong P_{i_1}$, then
$$\rankv F^+_{i_1}(M) =s_{i_1}(c_\mathbf{i}^{-r}(\beta_{\mathbf{i},k}))=\left\{\begin{array}{ll}
                                                     c_{\mathbf{i}'}^{-r}(\beta_{\mathbf{i}',k-1})=\rankv\tau^{-r}_{H'}P'_{i_{k}}, &
                                                      \text{ if }  2\leq k\leq n, r\geq 0; \\ \\
                                                      c_{\mathbf{i}'}^{-r+1}(\beta_{\mathbf{i}',n})=\rankv \tau^{-r+1}_{H'}P'_{i_1}, &
                                                       \text{ if } k=1, r>0.
\end{array}\right.
$$
Note that following Proposition \ref{lem6}, $F^+_{i_1}(M)$ is $\tau$-locally free.
Therefore, $F^+_{i_1}(M)$ is  preprojective by Proposition \ref{rem1},
since its rank vector is the rank vector of a preprojective module.
Furthermore, by Proposition \ref{lem4} (3) that any two indecomposable preprojective $H$-modules with the same
rank vector are isomorphic,
$$F^+_{i_1}(M)\cong\left\{ \begin{array}{ll}
      \tau^{-r}_{H'}P'_{i_{k}}, & \text{ if } 2\leq k\leq n, r\geq0; \\\\
      \tau^{-r+1}_{H'}P'_{i_1}, &  \text{ if } k=1,r>0.
     \end{array}\right.
$$

(2) can be proved by similar arguments, since
$$s_{i_1}(c_\mathbf{i}^{s}(\gamma_{\mathbf{i},k}))=\left\{\begin{array}{ll}
                                                     c_{\mathbf{i}'}^{s}(\gamma_{\mathbf{i}',k-1}), &  \text{ if } 2\leq k\leq n,  s\geq0;\\\\
                                                      c_{\mathbf{i}'}^{s+1}(\gamma_{\mathbf{i}',n}), &  \text{ if }  k=1, s\geq0.
\end{array}\right.
$$

(3)
As $M$ is regular, for any $m\in\mathbb{Z}$, $\tau_{H}^{m}(M)$ is regular and so is not isomorphic to $E_{i_1}$. We have
\begin{equation}\label{eq:taurefl} F^+_{i_1}(\tau_{H}^{m}(M))=\tau_{H'}^{m}(F^+_{i_1}(M)).\end{equation}
 Consequently,
$ \tau_{H'}^{m}(F^+_{i_1}(M))$ is non-zero for any $m\in \mathbb{Z}$. Therefore $F^+_{i_1}(M)$ is regular, and further
$\tau$-locally free by Proposition \ref{lem6}.

(4) Following (\ref{eq:taurefl}), we have
\[
\tau_{H'}^{m}(F^+_{i_1}(M))\cong F^+_{i_1}(M) ~\text{ if and only if }~
\tau^{m}_H(M)\cong M.
\]
Consequently, if $M$ is contained in a  tube of rank $r$, so is $F^+_{i_1}(M)$.

Now assume that $M$ is a mouth module. Consider the AR-sequence
\[
\begin{tikzpicture}[scale=0.6,
 arr/.style={black, -angle 60}]
 \path (-1,1) node (c1) {$\mathcal{E}\colon$};
\path (3,1) node (c1) {$\tau M$};
\path (7,1) node (c2) {$X$};
\path (11,1) node (c3) {$M$};
\path (14, 1) node (c4) {$0.$};
\path (0, 1) node (c0) {$0$};

\draw[->] (c1) edge node[auto]{} (c2);
\draw[->] (c2) edge node[auto]{}  (c3);
\draw[->] (c0) edge (c1);
\draw[->] (c3) edge (c4);
\end{tikzpicture}\]
By assumption,  $\tau M$, $X$, $M$ are indecomposable regular modules and
so none of them is isomorphic to $E_{i_1}$. Therefore,
\[
\begin{tikzpicture}[scale=0.6,
 arr/.style={black, -angle 60}]
 \path (-1.4,1) node (c1) {$F_{i_1}^+\mathcal{E}\colon$};
\path (3,1) node (c1) {$F_{i_1}^+ \tau M$};
\path (7,1) node (c2) {$F_{i_1}^+ X$};
\path (11,1) node (c3) {$F_{i_1}^+ M$};
\path (14, 1) node (c4) {$0$};
\path (0, 1) node (c0) {$0$};

\draw[->] (c1) edge node[auto]{} (c2);
\draw[->] (c2) edge node[auto]{$f$}  (c3);
\draw[->] (c0) edge (c1);
\draw[->] (c3) edge (c4);
\end{tikzpicture}\]
is exact and all the three modules are indecomposable. Furthermore,  $F^+_{i_1}\mathcal{E}$
is an AR-sequence and so $F^+_{i_1}M$ is a mouth module. Indeed,  by the fact that $F^+_{i_1}$
induces an equivalence from $\mathrm{rep}_{\mathrm{lf}}H \backslash \mathrm{add} E_{i_1}$ to
$\mathrm{rep}_{\mathrm{lf}}H'\backslash \mathrm{add} E'_{i_1}$,
we need only to show that any map from the generalised simple $H'$-module $E'_{i_1}$ to $F^+_{i_1}M$ factors
through $f$, but this is trivial, since, by Lemma \cite[Lemma 11.7]{[GLS1]}, any map from
an injective module to a $\tau$-locally free regular module is zero. 
Note that $i_1$ is now a source with respect to  $s_{i_1}(\Omega)$ and so $E'_{i_1}$ is an injective 
$H'$-module.
This completes the proof.
\end{proof}

By similar computation as in Proposition  \ref{lem 2.1}, we have the following.

\begin{prop}\label{dualcor2}  \qquad
\begin{itemize}
\item[(1)] The reflection functor $F^-_{i_1}$ maps indecomposable preprojective or preinjective $H'$-modules 
that are not isomorphic to $E'_{i_1}$ to  indecomposable preprojective or preinjective $H$-modules that 
are not isomorphic to  $E_{i_1}$.

\item[(2)] The reflection functor $F^-_{i_1}$ maps $\tau$-locally free regular $H'$-modules to
$\tau$-locally free regular $H$-modules. Moreover, $F^-_{i_1}$ preserves tubes and mouth modules.
\end{itemize}
\end{prop}

\begin{prop}\label{prop:indoforiet}
Let  $\Omega'$ be obtained from $\Omega$ by reflection at an admissible vertex $i$, $H=H(C,D,\Omega)$ and
$H'=H(C,D,\Omega')$.
\begin{itemize}
\item[(1)] Suppose that positive roots of type $C$ are rank vectors of some $\tau$-locally free $H$-modules.
Then so is true for $H'$.
\item[(2)] Suppose that  rank vectors of $\tau$-locally free $H$-modules are positive roots of type $C$. Then so is true for $H'$.
\end{itemize}
Consequently, if Conjecture \ref{Conj:gls} holds for $H$, then so it does  for $H'$.
\end{prop}

\begin{proof}
Without loss of generality we may assume that the admissible vertex is  a sink.
Recall the simple fact  that 
\begin{equation}\label{eq:simple}\rankv E_{i}=\alpha_{i}=\rankv E'_{i}.
\end{equation}

(1) Let $\alpha$ be a positive root. By (\ref{eq:simple}),  we may assume that $\alpha\not=\alpha_i$. 
Then $s_i(\alpha)$ is also a positive root. 
By assumption, there exists a $\tau$-locally free $H$-module $M$ such that $\rankv M=s_i(\alpha)$. 
Note that $s_i(\alpha)$ is different from $\alpha_i$ and so $M\not\cong E_i$. 
Now by Proposition~\ref{lem6}, $F_i^+(M)$ is a $\tau$-locally free  $H'$-module with 
\[\rankv F_i^+(M)=s_i(\rankv M)=s_i(s_i(\alpha))=\alpha.\]
This proves (1).

(2) Let $L$ be a $\tau$-locally free $H'$-module. By  (\ref{eq:simple}) and the fact that
$E'_i$ is $\tau$-locally free, we may assume that $L\not\cong E'_i$. By 
Proposition~\ref{lem6}, $F_i^-L$ is a $\tau$-locally free $H$-module with 
 $\rankv  F_i^-L=s_i(\rankv L)$, which is  a root by assumption. 
 Therefore, \[\rankv L=s_i(\rankv F_i^-L)\] is also a root. This proves (2).
\end{proof}

\subsection{Conjecture \ref{Conj:gls}$\colon$the forward implication}
Let $A=KQ/I$, and $a,b\in Q_0$. Denote by $r_{\rho}: Ae_a\rightarrow Ae_b$ (resp. $l_{\rho}: e_bA\rightarrow e_aA$)
the morphism of projective left (resp. right) $A$-modules given by the right (resp. left) multiplication by a path $\rho$ from $b$ to $a$. Recall that
$$\nu P(a)=I(a),$$
where $\nu=D\text{Hom}_A(-,A)$ is the Nakayama functor.
Since the assignment $f\mapsto f(e_a)$ induces  an isomorphism $\eta_a: \text{Hom}_A(Ae_a,A)\cong e_aA$ of right $A$-modules, we have the following  commutative diagram.

$$\xymatrixcolsep{5pc}\xymatrixrowsep{2pc}\xymatrix{
\text{Hom}_A(Ae_b,A)\ar[r]^{\text{Hom}_A(r_{\rho},A)}\ar[d]_{\cong}^{\eta_b} & \text{Hom}_A(Ae_a,A) \ar[d]_{\cong}^{\eta_a} \\
e_bA \ar[r]^{l_{\rho}} & e_aA\\
}$$
Therefore, we can identify the morphism $\nu(r_{\rho}): I(a)\rightarrow I(b)$ as $D(l_{\rho})=l_{\rho}^*$.
Note that for any $x\in e_aA$, we denote by $x^*\in I(a)$ the dual of $x$.
\begin{lem} The map $l_{\rho}^*$ can be described as follows.
\begin{align}\label{eqn:1}  l_{\rho}^*(x^*)=\left\{\begin{array}{ll}
     y^*, & \text{if}\ x=\rho y, x\in e_aA \text{ and } y\in e_b A; \\
     0, & \text{otherwise}.
    \end{array}\right.
\end{align}
\end{lem}
\begin{proof}
Note that $l_{\rho}^*(x^*)(y)=x^*l_{\rho}(y)=x^*(\rho y)$ and so the lemma follows.
\end{proof}

The following useful fact for computing $\tau M$ is an immediate consequence of (\ref{eqn:1}).
\begin{prop}\label{Prop:tau}
Let
\[
\begin{tikzpicture}[scale=0.7,
 arr/.style={black, -angle 60}]
\path (11,1) node (c3) {$M$};
\path (7,1) node (c2) {$\oplus_{t=1}^nP_{j_t}$};
\path (3,1) node (c1) {$\oplus_{s=1}^mP_{i_s}$};
\path (14, 1) node (c4) {$0$};
\path (0, 1) node (c0) {$0$};

\draw[->] (c1) edge node[auto]{$r$} (c2);
\draw[->] (c2) edge node[auto]{}  (c3);
\draw[->] (c0) edge (c1);
\draw[->] (c3) edge (c4);
\end{tikzpicture}\]
 be a  projective resolution of an  $A$-module $M$ with $r=(r_{\rho_{st}})$.  Then
 \[
\begin{tikzpicture}[scale=0.7,
 arr/.style={black, -angle 60}]
\path (11,1) node (c3) {$\oplus_{t=1}^nI_{j_t}$};
\path (7,1) node (c2) {$ \oplus_{s=1}^mI_{i_s}$};
\path (3,1) node (c1) {$\tau(M)$};
\path (14, 1) node (c4) {$0$};
\path (0, 1) node (c0) {$0$};

\draw[->] (c1) edge node[auto]{} (c2);
\draw[->] (c2) edge node[auto]{$\nu(r)$}  (c3);
\draw[->] (c0) edge (c1);
\draw[->] (c3) edge (c4);
\end{tikzpicture}\]
 is an  injective resolution of $\tau(M)$ 
with $\nu(r)=(l_{\rho_{st}}^*)$.
\end{prop}

We say that  {\em a root is preprojective, preinjective} and {\em regular} if it is the rank vector of a preprojective, preinjective and regular module, respectively.

\begin{prop}\cite[Proposition 1.9]{[DR]} \label{prop:roots} Let $\mathbf{i}=(i_1, \dots, i_n)$ be a
$+$-admissible sequence with respect to  $(C,\Omega)$.
The set of  positive roots of type $C$ is  the disjoint union of preprojective, preinjective and regular roots as follows.
\begin{itemize}
\item[(1)] $\{c_\mathbf{i}^{-r}(\beta_{\mathbf{i},k})\mid r\in\mathbb{Z}_{\geq 0},1\leq k\leq n\}.$

\vspace{1mm}

\item[(2)] $\{c_\mathbf{i}^{s}(\gamma_{\mathbf{i},k})\mid s\in\mathbb{Z}_{\geq 0},1\leq k\leq n\}.$

\vspace{1mm}

\item[(3)] $\{x_0+ rg\delta\mid x_0 \textup{ is  a positive root }\leq g\delta, ~
r\in\mathbb{Z}_{\geq 0}\}$,
where  $1\leq g\leq 3$ is a constant and $x_0$ can be deduced from Table 6 in \cite{[DR]}.
\end{itemize}
\end{prop}

We remark that the set of roots in Proposition \ref{prop:roots} (3) contains all the positive imaginary
roots $r\delta$, where $r\geq 1$.

We say that a vector appearing in a component of an AR-quiver, if it is the dimension or rank vector
of an indecomposable representation contained in the component.
By saying that $$\xymatrix{M_r& \cdots & M_2&M_1}$$ are {\em the mouth modules of a tube of rank} $r$, we mean that
they are  non-isomorphic mouth modules,
$\tau M_i\cong M_{i+1}$ for any $1\leq i\leq r-1$ and $\tau M_r\cong M_1$.

\begin{prop}\label{Prop:typeA}
Let  $C$ be the Cartan matrix of type $\type A_n$. Then
for any positive root $\alpha$ of  type $C$ and any symmetriser $D$,
there exists a $\tau$-locally free $H$-module $M$ such that $\rankv M=\alpha$.
\end{prop}
\begin{proof} First note that, by Lemma \ref{lem:rkvpreprojinj} and Proposition \ref{Lem:tauc},
\[
\rankv \tau^{-r} P_{i_k}=c^{-r} \beta_{k} \text{ and } \rankv \tau^{r} I_{i_k}=c^{r} \gamma_{k}.
\]
So by Proposition \ref{prop:roots}, we only need to consider the case where $\alpha$ is a regular root.
When $n=1$, the only regular roots are the positive multiples of $\delta$.
By duality, we may assume that the orientation $\Omega =\{(2, 1)\}$. In this case,
the result follows from Proposition \ref{Prop:homog}.

Next assume that $n>1$. Let $D=(m, \dots, m)$.
Assume that the vertices $1, \dots, n$, are clock-wise cyclically ordered.
Let $M(i, j)$ be the locally free representation supported at the cyclic interval $[i, j]$ with $H_i$
at the vertices in $[i, j]$ and  the identity map on each arrow within the support, and let
 $r_{ij}=\rankv M(i, j)$.
Note that $H_i\cong K[x]/(x^m)$. Hence \[\dim \End M(i, j)=m,\] and so
when $j\not=i-1$, i.e. $[i, j]\not=Q_0$, by Proposition \ref{lem2.2}, $M(i, j)$ is rigid.
Consequently, $M(i, j)$ is $\tau$-locally free,
and so by Theorem \ref{thm:Schurroot},
$r_{ij}$ is a real Schur root  and $M(i, j)$ is the unique $\tau$-locally free module
with $r_{ij}$ as the rank vector.

The AR-quiver $\Gamma_{KQ^0}$ has one or two inhomogeneous tubes, depending on the
orientation. Let $\mathcal{C}$ be such a tube.
The dimension vectors of its mouth modules are
roots of the form $r_{ij}$ with  $[i, j]\not=Q_0$ and assume that they are
$r_{i_sj_s}, \dots, r_{i_1j_1}$. Since
\[
\rankv \tau^s M(i_1, j_1)=c^s(r_{i_1 j_1})=r_{i_1 j_1}=\rankv M(i_1, j_1),\]
by Theorem \ref{thm:Schurroot}, we have \[\tau^s M(i_1, j_1)\cong M(i_1, j_1).\]
So  $M(i_s, j_s), \dots, M(i_1,  j_1)$ are
contained in a tube $\mathcal{T}$ of $\Gamma_H$ by Theorem \ref{lem1}, and they form a $\tau$-orbit.
Note  that the minimal regular real Schur roots appear either at the bottom of $\mathcal{C}$ or
the other inhomogeneous tube of $\Gamma_{kQ^0}$ if it exists,
and that any proper nonzero submodule of $M(i, j)$ is again of the form $M(i', j')$, whose
rank vector is then also a real Schur root and is smaller than $r_{ij}$. So  such a submodule
 $M(i', j')$ can not appear in $\mathcal{T}$.
Therefore,
$M(i_s, j_s), ~\dots, ~M(i_2, j_2), ~ M(i_1, j_1)$
are the mouth modules of
the rank $s$ tube.
Hence $\mathcal{T}$ is  an inhomogeneous tube such that the
rank vectors appearing in it are the same as the roots appearing in $\mathcal{C}$.
As any regular root appears in an inhomogeneous tube, for any given regular root $\alpha$,
there exists a $\tau$-locally free $H$-module such that its rank vector is $\alpha$.
This completes the proof.
\end{proof}

\begin{lem}\label{lem0}
If the $c$-period of $\alpha_i$ is $r$ for $i\in Q_0$, then
$$\xymatrix{\tau^{r-1}E_i &\cdots &\tau E_i & E_i}$$ are rigid $\tau$-locally free $H$-modules and the mouth modules of
a rank $r$ inhomogeneous tube.
\end{lem}

\begin{proof}
First note that the modules are rigid, since each $E_i$ is rigid and $\tau$ preserves rigidity.
Next,
\[\rankv \tau^r E_i=c^r\alpha_i=\alpha_i=\rankv E_i,\]
and so $\tau^r E_i\cong E_i$ by Theorem \ref{thm:Schurroot}. 
Therefore, $\tau^{r-1} E_i, ~ \dots, ~ \tau E_i, ~ E_i$
are the mouth modules of an inhomogeneous tube of $\tau$-locally free modules.
\end{proof}

Note that any proper subquiver of $Q^0$ is a Dynkin quiver.
Let $e\in Q_0$ such that $Q^0\backslash \{e\}$ is connected. In particular,
there is exactly one arrow incident at $e$ (assuming $Q_0$ has at least two vertices).
We call such an $e$ a {\em boundary vertex}.
 Denote the quotient algebra $H/(e)$ by $\delE$ and
the AR-translation in $\rep (\delE)$ by $\tau'$,
 where $(e)$ is the ideal generated by $e$.

\begin{lem}\label{lem0.1}
Let $M$ be a $\tau$-locally free $H$-module with $M_e=0$ for a boundary vertex $e$. Then  $M$ is also a $\tau$-locally free $\delE$-module. Moreover,
$\rankv M$ is a positive root of type $C$.
\end{lem}

\begin{proof}
As $e$ is boundary, it is always admissible. By duality,  we  may assume that $\mathbf{i}=(i_1, i_2, \dots, i_{n-1},i_n=e)$
is a $+$-admissible sequence.  Following Proposition \ref{lem:TC},
\[
\tau'M\cong T' F^+_{i_{n-1}}\dots F^+_{i_2} F^+_{i_1}M,
\]
where $F^+_{i_{n-1}}\dots F^+_{i_2} F^+_{i_1}$ is restricted to $\text{rep}(H')$ and $T'$ is the twist automorphism.
 Now by Proposition \ref{lem6},  $(\tau')^k M$ is either zero or
(non-zero) locally free for any $k\in\mathbb{Z}$.  So $M$ is a $\tau$-locally free $H'$-module.  Thus  its  rank vector is  a root
of the Dynkin subquiver $Q^0\backslash \{e\}$, by \cite[Theorem 1.3]{[GLS1]},  and so  is a root of type $C$.
\end{proof}

For any $C$, denote by $\val$ the corresponding valued quiver in \cite[Section 6]{[DR]}.
We say that an inhomogeneous tube $\mathcal{T}$ of rank $r$ is {\em good} if the mouth modules are rigid.  

\begin{lem}\label{lem0.2}
Let  $H=H(C,D,\Omega)$ and $H'=H(C,mD,\Omega)$, where $D$ is minimal, $m>1$ and the orientation $\Omega$
is the same as that of $\val$.
Suppose that for any inhomogeneous tube $\mathcal{C}$ of rank $r$ in the AR-quiver of $\val$,
there exists a good tube $\mathcal{T}$ of $\tau$-locally free $H$-modules of the same rank such that
\begin{equation}\label{eq:rankmouth}
\{\rankv M| \ M\in \mathcal{T} \text{ is a mouth module}\}=
\{\dimv M| \ M \in \mathcal{C} \text{ is a mouth module}\}.
\end{equation}
Then
there exists a good tube $\mathcal{T}'$ of  $\tau$-locally free $H'$-modules that satisfies (\ref{eq:rankmouth}).
\end{lem}
\begin{proof} When $C$ is of type $\type{A}_n$, the result follows from Proposition \ref{Prop:typeA}. So below we only consider
those $C$ that are of other types.
There exists at least one root $\alpha$ that appears at the bottom of  $\mathcal{C}$ and
is supported at a Dynkin subquiver. It follows that $\alpha$ is a real Schur root by (\ref{eq:rankmouth}) and  Theorem \ref{thm:Schurroot}.
Again by (\ref{eq:rankmouth}) and Theorem \ref{thm:Schurroot}, for the symmetriser $mD$, there exists a unique rigid $\tau$-locally free $H'$-module
$M'$ such that $\rankv M'=\alpha$, and $M'$ is $\tau$-periodic with the $\tau$-period $r$. So by Proposition \ref{cor}, $M'$ is contained in a tube $\mathcal{T}'$ of rank $r$.
We claim that $M'$ is a mouth module.
Indeed, otherwise $M'$ would have a  $\tau$-locally free mouth submodule $N'$ in $\mathcal{T}'$, 
whose rank vector
would then also be a real Schur root $\beta$, by Lemma \ref{lem0.1}. Moreover, $r$ would be the $c$-period of $\beta$ and
$\beta<\alpha$.
Therefore, $\beta$ must be from another inhomogeneous tube of the same rank  in the list of \cite[Section 6]{[DR]}. We check the existence of $\beta$ and we only need to do that when $C$ is of type $\type{BD}_3$, $\type{CD}_3$, $\type{D}_n$ and $\type{E}_6$, since the ranks of distinct inhomogeneous tubes are different for other types. If $C$ is of type $\type{E}_6$,  there are two inhomogeneous tubes of rank 3 in $\Gamma_H$, whose mouth modules are all 4-dimensional. So no such $\beta$ exists. If $C$ is of type $\type{D}_n$ for $n\neq 4$, similarly no such $\beta$ exists. If $C$ is of type $\type{D}_4$,  there are three inhomogeneous tubes of rank 2.
 For the other two types, there are two inhomogeneous tubes of rank 2.
Observe that for all the three cases, there exists  a real Schur root $\gamma$ such that $\gamma<\alpha$, but such a $\gamma$ is not the rank vector of any $\tau$-locally free regular submodule of $M'$. So  such a submodule $N'$ does not exist.  Therefore, $M'$ is a mouth module, and so $\mathcal{T}'$ satisfies (\ref{eq:rankmouth}).
\end{proof}

\begin{thm}\label{thm3.1}
Let $C$ be a Cartan matrix of affine type and $D$ an arbitrary symmetriser of $C$. 
Then for any positive root $\alpha$ of $g(C)$,
there exists a $\tau$-locally free $H$-module $M$ such that $\rankv M=\alpha$.
\end{thm}

\begin{proof} By Proposition \ref{Prop:typeA}, we may assume that $C$ is of any affine type but
$\widetilde{\mathbb{A}}_n$ and so $Q^0$ is a tree. This means that any two orientations
can be obtained from each other by reflections (see  \cite[Lemma 5.2]{ASS}).
Now by Proposition \ref{prop:indoforiet} (1), we need only to prove the theorem for a fixed orientation. We
will use the same orientations as $\val$ in \cite{[DR]}, where Dlab-Ringel proved that
there is a bijection between  dimension vectors of indecomposable representations of
$\val$ and positive roots of  $g(C)$.

As in type $\type{A}_n$,  to prove the theorem, it suffices to
show that for each regular root in Proposition \ref{prop:roots} (3), there exists a $\tau$-locally free module such that the rank vector is the root.
Observe from the lists in \cite[Section 6]{[DR]} and Proposition \ref{prop:roots} (3) that  all the regular roots appear in an inhomogeneous tube when $C$ is not of type $\type B_n$ and
$\type G_{21}$, while for these two types so do all the regular roots but some positive multiples of $\delta$,
which  however appear in some homogeneous tube, following Proposition \ref{Prop:homog}.
So it remains to prove the following.

{\bf Claim}: {\em for each inhomogeneous tube $\mathcal{C}$ in \cite[Section 6]{[DR]}, there is a  good tube of $\tau$-locally free $H$-modules of the same rank such that the rank vectors of the mouth modules are exactly the same as the dimension vectors of the mouth modules of $\mathcal{C}$.}

Let $r$ be the rank of $\mathcal{C}$.
We prove the claim in the following few cases.

(1) The tube $\mathcal{C}$ contains a simple representation $S_i$ of the valued quiver
$\val$.  Then
$c^r\alpha_i=\alpha_i.$
By Lemma \ref{lem0},
$\tau^{r-1} E_i, ~ \dots, ~\tau E_i, ~ E_i$
are the mouth modules of a rank $r$ tube of $\tau$-locally free $H$-modules.
The rank vectors of the mouth modules
are exactly the same as the dimension vectors of the mouth modules in $\mathcal{C}$.
Thus the claim holds for such a tube $\mathcal{C}$.

(2) The Cartan matrix $C$ is of simply laced type  $\type{D}_n$, $\type{E}_{6}$, $\type{E}_{7}$ and $\type{E}_{8}$. When $D$ is
minimal, $H$ is the path algebra of the corresponding quiver.
From the representation theory of  affine quivers, we know that all the inhomogeneous tubes in $\Gamma_H$ are good.
So by Lemma \ref{lem0.2}, the claim holds in this case.

(3) Finally, for the remaining inhomogeneous tubes, by Lemma \ref{lem0.2},  the tubes constructed in Propositions \ref{typeBD2}, \ref{typeCD2} to \ref{typeG2} complete the proof of the claim.
\end{proof}

\section{Homogeneous tubes of type $\type{A}_{11}, \type{A}_{12}$,
$\type{B}_{n}$ and $\type{G}_{21}$}\label{Sec:homogtubes}

In this section, $C$ is of type $\type{A}_{11}, \type{A}_{12}$, $\type{B}_{n}$ and $\type{G}_{21}$.
For each type, we construct an indecomposable $\tau$-locally free $H$-module $M$ such that
\begin{equation}\label{eq:homog}
\tau M=M ~\text{ and }~ \rankv M=\delta.
\end{equation}
Consequently, by Theorem \ref{lem1},
such a module  is contained in a homogenous tube.
Recall from  \cite[Corollary 4.3]{[GLS1]} that
the bilinear form $(-, -)_C$ defined by $C$ is the symmetrisation of the one in (\ref{eq:bil}) and
that   by \cite[Lemma 1.3]{[DR]} and the discussion after it,
\[\{d\in \mathbb{Z}^{n}| (d, y)_C=0\ \text{for\ any}  ~y \in \mathbb{Z}^n\}=\mathbb{Z}\delta.\]
  Therefore $M$ has
no non-zero proper submodules that are in the same tube as $M$, and so $M$ is a mouth module.
Consequently, we have the following.

\begin{prop}\label{Prop:homog}
Any positive multiple of $\delta$ appears in a homogeneous tube of $\tau$-locally free $H$-modules.
\end{prop}

\subsection{Type $\type{A}_{11}$}\label{sec:a11}
In this case,  $C=\left(
\begin{array}{cc}
2 & -4\\
-1 & 2
\end{array}
\right)$, $D=\text{diag}(m,4m)$. Let $\Omega ={(2,1)}$. Then  $H=H(C,D,\Omega)$ is given by the quiver
\[\begin{tikzpicture}[scale=1.1]
\node (-20) at (0,0) {$1$};
\node (-20) at (1.5,0) {$2$};
\draw [->] (0.25,0) -- (1.25,0);
\node (-20) at (0.75,-0.2) {$a$};
\node (-20) at (1.5,0.8) {$\varepsilon_2$};
\node (-20) at (0,0.8) {$\varepsilon_{1}$};
\draw[-latex] (-0.2,0.1) .. controls (-0.5,0.7) and (0.5,0.7) .. (0.2,0.05);
\draw[-latex] (1.3,0.1) .. controls (1,0.7) and (2,0.7) .. (1.7,0.05);
\end{tikzpicture}\]
with relations $\varepsilon_1^m=0$, $\varepsilon_2^{4m}=0$ and $a\varepsilon_1=\varepsilon_2^4a$.
The minimal positive imaginary root $\delta=(2, 1)$.

The following  module $M$ is indecomposable and  satisfies (\ref{eq:homog}).
$$\xymatrixcolsep{1pc}\xymatrixrowsep{1pc}\xymatrix{
&1\ar[dd]\ar[rrrr]&&&& 1\ar[dd]\ar[rr]& & \cdots \ar[rr]&& 1\ar[dd]\ar[rrrr]&&&&1\ar[dd]\\
1\ar[d]\ar[rrrr]&&&& 1\ar[d]\ar[rr]& & \cdots \ar[rr]&& 1\ar[d]\ar[rrrr]&&&&1\ar[d] \\
2\ar[r] & 2\ar[r] & 2\ar[r] & 2\ar[r] & 2\ar[r] & 2\ar[r] & \cdots \ar[r] & 2\ar[r] & 2\ar[r] &2\ar[r] & 2\ar[r]& 2\ar[r] & 2\ar[r] &2\ar[r] & 2\ar[r]&2\\
}$$
Indeed, first by observation $\End M\cong K$ and so $M$ is indecomposable. Next consider
the following projective resolution of $M$,
\[
\begin{tikzpicture}[scale=0.6,
 arr/.style={black, -angle 60}]
 \path (0, 1) node (c0) {$0$};
\path (3,1) node (c1) {$P_2$};
\path (7,1) node (c2) {$P_1\oplus P_1$};
\path (11,1) node (c3) {$M$};
\path (14, 1) node (c4) {$0$.};

\draw[->] (c1) edge node[auto]
{$\left(\begin{smallmatrix}	r_a\\ -r_{\varepsilon_2a}\\ \end{smallmatrix}\right)$} (c2);
\draw[->] (c2) edge node[auto]{}  (c3);
\draw[->] (c0) edge (c1);
\draw[->] (c3) edge (c4);
\end{tikzpicture}\]	
 Then by Proposition \ref{Prop:tau},
\[
\begin{tikzpicture}[scale=0.6,
 arr/.style={black, -angle 60}]
 \path (0, 1) node (c0) {$0$};
\path (3,1) node (c1) {$\tau M$};
\path (7,1) node (c2) {$I_2$};
\path (11,1) node (c3) {$I_1\oplus I_1$};
\path (14, 1) node (c4) {$0$};

\draw[->] (c1) edge node[auto]{} (c2);
\draw[->] (c2) edge node[auto]{$\left(\begin{smallmatrix} l^*_a\\ -l_{\varepsilon_2a}^*\\
\end{smallmatrix}\right)$}  (c3);
\draw[->] (c0) edge (c1);
\draw[->] (c3) edge (c4);
\end{tikzpicture}\]
is a short exact sequence. By direct computation following (\ref{eqn:1}), we have  $\tau M\cong M$. So 
$M$ is $\tau$-locally free with $\rankv M=\delta$. 

\subsection{Type {$\type{A}_{12}$}}\label{sec:a12}
In this case $C=\left(
\begin{array}{cc}
2 & -2\\
-2 & 2
\end{array}
\right)$, $D=\text{diag}(m,m)$. Let $\Omega ={(2,1)}$. Then  $H=H(C,D,\Omega)$ is given by the quiver
\[\begin{tikzpicture}[scale=1.1]
\node (-20) at (0,0) {$1$};
\node (-20) at (1.5,0) {$2$};
\draw [->] (0.3,0.1) -- (1.25,0.1);
\draw [->] (0.3,-0.1) -- (1.25,-0.1);
\node (-20) at (0.75,0.25) {$a_1$};
\node (-20) at (0.75,-0.25) {$a_2$};
\node (-20) at (1.5,0.8) {$\varepsilon_2$};
\node (-20) at (0,0.8) {$\varepsilon_{1}$};
\draw[-latex] (-0.2,0.1) .. controls (-0.5,0.7) and (0.5,0.7) .. (0.2,0.05);
\draw[-latex] (1.3,0.1) .. controls (1,0.7) and (2,0.7) .. (1.7,0.05);
\end{tikzpicture}\]
with relations $\varepsilon_i^m=0$ and $a_i\varepsilon_1=\varepsilon_2a_i$ for $i=1,2$.
We have $\delta=(1, 1)$.

The following  module $M$ is  indecomposable and satisfies  (\ref{eq:homog}).

\[\begin{tikzpicture}
\node (-20) at (0,0) {$1$}; \node (-20) at (1.5,0) {$1$}; \node (-20) at (3,0) {$\cdots$};
\node (-20) at (4.5,0) {$1$}; \node (-20) at (6,0) {$1$};
\node (-20) at (0,-1.2) {$2$}; \node (-20) at (1.5,-1.2) {$2$}; \node (-20) at (3,-1.2) {$\cdots$};
\node (-20) at (4.5,-1.2) {$2$}; \node (-20) at (6,-1.2) {$2$};
\draw [->] (0.3,0) -- (1.25,0); \draw [->] (1.8,0) -- (2.65,0);
\draw [->] (3.4,0) -- (4.25,0); \draw [->] (4.8,0) -- (5.75,0);
\draw [->] (0.3,-1.2) -- (1.25,-1.2); \draw [->] (1.8,-1.2) -- (2.65,-1.2);
\draw [->] (3.4,-1.2) -- (4.25,-1.2); \draw [->] (4.8,-1.2) -- (5.75,-1.2);
\draw [->] (-0.07,-0.25) -- (-0.07,-0.95); \draw [->] (0.07,-0.25) -- (0.07,-0.95);
\draw [->] (1.43,-0.25) -- (1.43,-0.95); \draw [->] (1.57,-0.25) -- (1.57,-0.95);
\draw [->] (4.43,-0.25) -- (4.43,-0.95); \draw [->] (4.57,-0.25) -- (4.57,-0.95);
\draw [->] (5.93,-0.25) -- (5.93,-0.95); \draw [->] (6.07,-0.25) -- (6.07,-0.95);
\node (-20) at (-0.2,-0.6) {$0$}; \node (-20) at (0.2,-0.6) {$1$};
\node (-20) at (1.3,-0.6) {$0$}; \node (-20) at (1.7,-0.6) {$1$};
\node (-20) at (5.8,-0.6) {$0$}; \node (-20) at (6.2,-0.6) {$1$};
\node (-20) at (4.3,-0.6) {$0$}; \node (-20) at (4.7,-0.6) {$1$};
\end{tikzpicture}\]
\noindent
Indeed, observe that $\End M\cong K[x]/(x^m)$ is a local ring and so $M$ is indecomposable.
Moreover, $M$ has the following projective resolution,
\[
\begin{tikzpicture}[scale=0.6,
 arr/.style={black, -angle 60}]
 \path (0, 1) node (c0) {$0$};
\path (3,1) node (c1) {$P_2$};
\path (7,1) node (c2) {$P_1$};
\path (11,1) node (c3) {$M$};
\path (14, 1) node (c4) {$0$.};

\draw[->] (c1) edge node[auto]
{$r_{a_1}$} (c2);
\draw[->] (c2) edge node[auto]{}  (c3);
\draw[->] (c0) edge (c1);
\draw[->] (c3) edge (c4);
\end{tikzpicture}\]
 By Proposition \ref{Prop:tau} and  (\ref{eqn:1})
\[\tau M= \ker l^*_{a_1}\cong M.\]
Therefore $M$ is a module as claimed.

\subsection{Type {$\type{B}_{n}$}}\label{sec:homogBn} In this case, $C$ is the following $(n+1)\times (n+1)$ matrix,
 $$\left(
                \begin{array}{ccccccc}
                  2 & -2 &  &  & & \\
                  -1 & 2 & -1 &  &  &\\
                 & & \ddots & \ddots & \ddots & \\
                 & & &  & -1 & 2 & -1 \\
                 & &  &  &  & -2 & 2 \\
                \end{array}
              \right)$$
and $D=\textup{diag}(m, ~2m, \cdots, ~2m, ~m)$. Let $\Omega=\{(2,1),(3,2),\cdots,(n+1,n)\}$.
Then $H=H(C,D,\Omega)$  is given by the quiver \[\begin{tikzpicture}[scale=1.1]
\node (-20) at (0,0) {$1$};
\node (-20) at (1.5,0) {$2$};
\node (-20) at (3,0) {$3$};
\node (-20) at (4.5,0) {$\cdots$};
\node (-20) at (6,0) {$n$};
\node (-20) at (7.8,-0.02) {$n+1$};

\node (-20) at (0.75,-0.2) {$a_{1}$};
\node (-20) at (2.25,-0.2) {$a_{2}$};
\node (-20) at (3.75,-0.2) {$a_{3}$};
\node (-20) at (5.25,-0.2) {$a_{n-1}$};
\node (-20) at (6.7,-0.2) {$a_{n}$};

\draw [->] (0.25,0) -- (1.25,0);
\draw [->] (1.75, 0) -- (2.75,0);
\draw [->] (3.25,0) -- (4.25,0);
\draw [->] (4.75, 0) -- (5.8,0);
\draw [->] (6.25, 0) -- (7.3,0);

\draw[-latex] (-0.2,0.1) .. controls (-0.5,0.7) and (0.5,0.7) .. (0.2,0.05);
\draw[-latex] (1.3,0.1) .. controls (1,0.7) and (2,0.7) .. (1.7,0.05);
\draw[-latex] (2.8,0.1) .. controls (2.5,0.7) and (3.5,0.7) .. (3.2,0.07);
\draw[-latex] (5.8,0.1) .. controls (5.5,0.7) and (6.5,0.7) .. (6.2,0.05);
\draw[-latex] (7.5,0.1) .. controls (7.2,0.7) and (8.2,0.7) .. (7.9,0.05);

\node (-20) at (0,0.8) {$\varepsilon_1$};
\node (-20) at (3,0.8) {$\varepsilon_3$};
\node (-20) at (1.5,0.8) {$\varepsilon_2$};
\node (-20) at (6,0.8) {$\varepsilon_n$};
\node (-20) at (7.7,0.8) {$\varepsilon_{n+1}$};
\end{tikzpicture}\]
with relations
$\varepsilon_i^{m}=0$ for $i=1,n+1$, $\varepsilon_j^{2m}=0$ for $2\leq j\leq n$, $a_1\varepsilon_1=\varepsilon_2^2a_1$, $a_n\varepsilon_n^2=\varepsilon_{n+1}a_n$ and $a_{k}\varepsilon_k=\varepsilon_{k+1}a_{k}$ for $2\leq k\leq n-1.$
We have $\delta=(1, 1, \dots, 1, 1)$ with all the entries equal to 1.

The projective module $P_1$ and the injective module $I_{n+1}$ are as follows,
and the other indecomposable projective and injective modules are submodules and quotient modules of
$P_1$ and $I_{n+1}$, respectively.

\[\begin{tikzpicture}[scale=1.1]
\node (-20) at (-1,0) {$P_1:$};
\node (-20) at (0,0) {$1$};
\node (-20) at (1,0) {$2$};
\node (-20) at (2,0) {$3$};
\node (-20) at (3.1,0) {$\cdots$};
\node (-20) at (4.5,0) {$n-1$};
\node (-20) at (5.7,0) {$n$};
\node (-20) at (7.05,0) {$n+1$};

\node (-20) at (0,-1.2) {$1$};
\node (-20) at (1,-0.6) {$2$};
\node (-20) at (2,-0.6) {$3$};
\node (-20) at (3.1,-0.6) {$\cdots$};
\node (-20) at (4.5,-0.6) {$n-1$};
\node (-20) at (5.7,-0.6) {$n$};
\node (-20) at (7.05,-1.2) {$n+1$};
\node (-20) at (7.8,-0.6) {$n+1$};

\node (-20) at (1,-1.2) {$2$};
\node (-20) at (2,-1.2) {$3$};
\node (-20) at (3.1,-1.2) {$\cdots$};
\node (-20) at (4.5,-1.2) {$n-1$};
\node (-20) at (5.7,-1.2) {$n$};
\node (-20) at (7.05,-1.2) {$n+1$};

\draw [->] (0.2,0) -- (0.8,0);
\draw [->] (1.2, 0) -- (1.8,0);
\draw [->] (2.2,0) -- (2.8,0);
\draw [->] (3.35, 0) -- (4,0);
\draw [->] (4.9,0) -- (5.4,0);
\draw [->] (5.95, 0) -- (6.5,0);

\draw [->] (1.2, -0.6) -- (1.8,-0.6);
\draw [->] (2.2,-0.6) -- (2.8,-0.6);
\draw [->] (3.35,-0.6) -- (4,-0.6);
\draw [->] (4.9,-0.6) -- (5.4,-0.6);
\draw [->] (5.9,-0.6) -- (7.3,-0.6);

\draw [->] (0.2,-1.2) -- (0.8,-1.2);
\draw [->] (1.2, -1.2) -- (1.8,-1.2);
\draw [->] (2.2,-1.2) -- (2.8,-1.2);
\draw [->] (3.35,-1.2) -- (4,-1.2);
\draw [->] (4.9,-1.2) -- (5.4,-1.2);
\draw [->] (5.95, -1.2) -- (6.5,-1.2);

\draw [->] (0,-0.2) -- (0,-1);
\draw [->] (1, -0.2) -- (1,-0.4);
\draw [->] (2,-0.2) -- (2,-0.4);
\draw [->] (4.5,-0.2) -- (4.5,-0.4);
\draw [->] (5.7,-0.2) -- (5.7,-0.4);
\draw [->] (7.05,-0.2) -- (7.05,-1);

\draw [->] (1, -0.8) -- (1,-1);
\draw [->] (2,-0.8) -- (2,-1);
\draw [->] (4.5,-0.8) -- (4.5,-1);
\draw [->] (5.7,-0.8) -- (5.7,-1);
\draw [->] (7.8,-0.8) -- (7.8,-1.6);

\draw [->] (0,-1.4) -- (0,-1.7);
\draw [->] (1, -1.4) -- (1,-1.7);
\draw [->] (2,-1.4) -- (2,-1.7);
\draw [->] (4.5,-1.4) -- (4.5,-1.7);
\draw [->] (5.7,-1.4) -- (5.7,-1.7);
\draw [->] (7.05,-1.4) -- (7.05,-1.7);

\node (-20) at (0,-1.85) {$\vdots$};
\node (-20) at (1,-1.85) {$\vdots$};
\node (-20) at (2,-1.85) {$\vdots$};
\node (-20) at (4.5,-1.85) {$\vdots$};
\node (-20) at (5.7,-1.85) {$\vdots$};
\node (-20) at (7.05,-1.85) {$\vdots$};
\node (-20) at (7.8,-1.85) {$\vdots$};

\draw [->] (0,-2.2) -- (0,-2.5);
\draw [->] (1, -2.2) -- (1,-2.5);
\draw [->] (2,-2.2) -- (2,-2.5);
\draw [->] (4.5,-2.2) -- (4.5,-2.5);
\draw [->] (5.7,-2.2) -- (5.7,-2.5);
\draw [->] (7.05,-2.2) -- (7.05,-2.5);
\draw [->] (7.8,-2.2) -- (7.8,-3);

\node (-20) at (0,-2.7) {$1$};
\node (-20) at (1,-2.7) {$2$};
\node (-20) at (2,-2.7) {$3$};
\node (-20) at (3.1,-2.7) {$\cdots$};
\node (-20) at (4.5,-2.7) {$n-1$};
\node (-20) at (5.7,-2.7) {$n$};
\node (-20) at (7.05,-2.7) {$n+1$};

\node (-20) at (0,-3.9) {$1$};
\node (-20) at (1,-3.3) {$2$};
\node (-20) at (2,-3.3) {$3$};
\node (-20) at (3.1,-3.3) {$\cdots$};
\node (-20) at (4.5,-3.3) {$n-1$};
\node (-20) at (5.7,-3.3) {$n$};
\node (-20) at (7.8,-3.3) {$n+1$};

\node (-20) at (1,-3.9) {$2$};
\node (-20) at (2,-3.9) {$3$};
\node (-20) at (3.1,-3.9) {$\cdots$};
\node (-20) at (4.5,-3.9) {$n-1$};
\node (-20) at (5.7,-3.9) {$n$};
\node (-20) at (7.05,-3.9) {$n+1$};
\node (-20) at (7.8,-4.5) {$n+1$};

\draw [->] (0.2,-2.7) -- (0.8,-2.7);
\draw [->] (1.2,-2.7) -- (1.8,-2.7);
\draw [->] (2.2,-2.7) -- (2.8,-2.7);
\draw [->] (3.35, -2.7) -- (4,-2.7);
\draw [->] (4.9,-2.7) -- (5.4,-2.7);
\draw [->] (5.95, -2.7) -- (6.5,-2.7);

\draw [->] (1.2, -3.3) -- (1.8,-3.3);
\draw [->] (2.2,-3.3) -- (2.8,-3.3);
\draw [->] (3.35,-3.3) -- (4,-3.3);
\draw [->] (4.9,-3.3) -- (5.4,-3.3);
\draw [->] (5.9,-3.3) -- (7.3,-3.3);

\draw [->] (0.2,-3.9) -- (0.8,-3.9);
\draw [->] (1.2, -3.9) -- (1.8,-3.9);
\draw [->] (2.2,-3.9) -- (2.8,-3.9);
\draw [->] (3.35,-3.9) -- (4,-3.9);
\draw [->] (4.9,-3.9) -- (5.4,-3.9);
\draw [->] (5.95, -3.9) -- (6.5,-3.9);

\draw [->] (0,-2.9) -- (0,-3.7);
\draw [->] (1, -2.9) -- (1,-3.1);
\draw [->] (2,-2.9) -- (2,-3.1);
\draw [->] (4.5,-2.9) -- (4.5,-3.1);
\draw [->] (5.7,-2.9) -- (5.7,-3.1);
\draw [->] (7.05,-2.9) -- (7.05,-3.7);

\draw [->] (1, -3.5) -- (1,-3.7);
\draw [->] (2,-3.5) -- (2,-3.7);
\draw [->] (4.5,-3.5) -- (4.5,-3.7);
\draw [->] (5.7,-3.5) -- (5.7,-3.7);
\draw [->] (7.8,-3.5) -- (7.8,-4.2);

\node (-20) at (1,-4.5) {$2$};
\node (-20) at (2,-4.5) {$3$};
\node (-20) at (3.1,-4.5) {$\cdots$};
\node (-20) at (4.5,-4.5) {$n-1$};
\node (-20) at (5.7,-4.5) {$n$};
\draw [->] (1, -4.1) -- (1,-4.3);
\draw [->] (2,-4.1) -- (2,-4.3);
\draw [->] (4.5,-4.1) -- (4.5,-4.3);
\draw [->] (5.7,-4.1) -- (5.7,-4.3);
\draw [->] (1.2, -4.5) -- (1.8,-4.5);
\draw [->] (2.2,-4.5) -- (2.8,-4.5);
\draw [->] (3.35,-4.5) -- (4,-4.5);
\draw [->] (4.9,-4.5) -- (5.4,-4.5);
\draw [->] (5.95, -4.5) -- (7.25,-4.5);
\end{tikzpicture}\]

\[\begin{tikzpicture}[scale=1.1]
\node (-20) at (-2,0.6) {$I_{n+1}:$};
\node (-20) at (1,0.6) {$2$};
\node (-20) at (2,0.6) {$3$};
\node (-20) at (3.1,0.6) {$\cdots$};
\node (-20) at (4.5,0.6) {$n-1$};
\node (-20) at (5.7,0.6) {$n$};
\draw [->] (-0.4, 0.6) -- (0.8,0.6);
\draw [->] (1.2, 0.6) -- (1.8,0.6);
\draw [->] (2.2,0.6) -- (2.8,0.6);
\draw [->] (3.35,0.6) -- (4,0.6);
\draw [->] (4.9,0.6) -- (5.4,0.6);
\draw [->] (1, 0.4) -- (1,0.2);
\draw [->] (2,0.4) -- (2,0.2);
\draw [->] (4.5,0.4) -- (4.5,0.2);
\draw [->] (5.7,0.4) -- (5.7,0.2);

\node (-20) at (-0.6,0.6) {$1$}; \draw [->] (-0.6,0.4) -- (-0.6,-0.4);
\node (-20) at (-0.6,-0.6) {$1$};
\draw [->] (-0.4,-0.6) -- (0.8,-0.6);
\draw [->] (-0.6,-0.8) -- (-0.6,-1.7);
\node (-20) at (-0.6,-1.8) {$\vdots$};
\draw [->] (-0.6,-2.2) -- (-0.6,-3.1);
\node (-20) at (-0.6,-3.3) {$1$};
\draw [->] (-0.4,-3.3) -- (0.8,-3.3);

\node (-20) at (0,0) {$1$};
\node (-20) at (1,0) {$2$};
\node (-20) at (2,0) {$3$};
\node (-20) at (3.1,0) {$\cdots$};
\node (-20) at (4.5,0) {$n-1$};
\node (-20) at (5.7,0) {$n$};
\node (-20) at (7.05,0) {$n+1$};

\node (-20) at (0,-1.2) {$1$};
\node (-20) at (1,-0.6) {$2$};
\node (-20) at (2,-0.6) {$3$};
\node (-20) at (3.1,-0.6) {$\cdots$};
\node (-20) at (4.5,-0.6) {$n-1$};
\node (-20) at (5.7,-0.6) {$n$};
\node (-20) at (7.05,-1.2) {$n+1$};

\node (-20) at (1,-1.2) {$2$};
\node (-20) at (2,-1.2) {$3$};
\node (-20) at (3.1,-1.2) {$\cdots$};
\node (-20) at (4.5,-1.2) {$n-1$};
\node (-20) at (5.7,-1.2) {$n$};
\node (-20) at (7.05,-1.2) {$n+1$};

\draw [->] (0.2,0) -- (0.8,0);
\draw [->] (1.2, 0) -- (1.8,0);
\draw [->] (2.2,0) -- (2.8,0);
\draw [->] (3.35, 0) -- (4,0);
\draw [->] (4.9,0) -- (5.4,0);
\draw [->] (5.95, 0) -- (6.5,0);

\draw [->] (1.2, -0.6) -- (1.8,-0.6);
\draw [->] (2.2,-0.6) -- (2.8,-0.6);
\draw [->] (3.35,-0.6) -- (4,-0.6);
\draw [->] (4.9,-0.6) -- (5.4,-0.6);

\draw [->] (0.2,-1.2) -- (0.8,-1.2);
\draw [->] (1.2, -1.2) -- (1.8,-1.2);
\draw [->] (2.2,-1.2) -- (2.8,-1.2);
\draw [->] (3.35,-1.2) -- (4,-1.2);
\draw [->] (4.9,-1.2) -- (5.4,-1.2);
\draw [->] (5.95, -1.2) -- (6.5,-1.2);

\draw [->] (0,-0.2) -- (0,-1);
\draw [->] (1, -0.2) -- (1,-0.4);
\draw [->] (2,-0.2) -- (2,-0.4);
\draw [->] (4.5,-0.2) -- (4.5,-0.4);
\draw [->] (5.7,-0.2) -- (5.7,-0.4);
\draw [->] (7.05,-0.2) -- (7.05,-1);

\draw [->] (1, -0.8) -- (1,-1);
\draw [->] (2,-0.8) -- (2,-1);
\draw [->] (4.5,-0.8) -- (4.5,-1);
\draw [->] (5.7,-0.8) -- (5.7,-1);

\draw [->] (0,-1.4) -- (0,-1.7);
\draw [->] (1, -1.4) -- (1,-1.7);
\draw [->] (2,-1.4) -- (2,-1.7);
\draw [->] (4.5,-1.4) -- (4.5,-1.7);
\draw [->] (5.7,-1.4) -- (5.7,-1.7);
\draw [->] (7.05,-1.4) -- (7.05,-1.7);

\node (-20) at (0,-1.85) {$\vdots$};
\node (-20) at (1,-1.85) {$\vdots$};
\node (-20) at (2,-1.85) {$\vdots$};
\node (-20) at (4.5,-1.85) {$\vdots$};
\node (-20) at (5.7,-1.85) {$\vdots$};
\node (-20) at (7.05,-1.85) {$\vdots$};

\draw [->] (0,-2.2) -- (0,-2.5);
\draw [->] (1, -2.2) -- (1,-2.5);
\draw [->] (2,-2.2) -- (2,-2.5);
\draw [->] (4.5,-2.2) -- (4.5,-2.5);
\draw [->] (5.7,-2.2) -- (5.7,-2.5);
\draw [->] (7.05,-2.2) -- (7.05,-2.5);

\node (-20) at (0,-2.7) {$1$};
\node (-20) at (1,-2.7) {$2$};
\node (-20) at (2,-2.7) {$3$};
\node (-20) at (3.1,-2.7) {$\cdots$};
\node (-20) at (4.5,-2.7) {$n-1$};
\node (-20) at (5.7,-2.7) {$n$};
\node (-20) at (7.05,-2.7) {$n+1$};

\node (-20) at (0,-3.9) {$1$};
\node (-20) at (1,-3.3) {$2$};
\node (-20) at (2,-3.3) {$3$};
\node (-20) at (3.1,-3.3) {$\cdots$};
\node (-20) at (4.5,-3.3) {$n-1$};
\node (-20) at (5.7,-3.3) {$n$};
\node (-20) at (7.05,-3.9) {$n+1$};

\node (-20) at (1,-3.9) {$2$};
\node (-20) at (2,-3.9) {$3$};
\node (-20) at (3.1,-3.9) {$\cdots$};
\node (-20) at (4.5,-3.9) {$n-1$};
\node (-20) at (5.7,-3.9) {$n$};

\draw [->] (0.2,-2.7) -- (0.8,-2.7);
\draw [->] (1.2,-2.7) -- (1.8,-2.7);
\draw [->] (2.2,-2.7) -- (2.8,-2.7);
\draw [->] (3.35, -2.7) -- (4,-2.7);
\draw [->] (4.9,-2.7) -- (5.4,-2.7);
\draw [->] (5.95, -2.7) -- (6.5,-2.7);

\draw [->] (1.2, -3.3) -- (1.8,-3.3);
\draw [->] (2.2,-3.3) -- (2.8,-3.3);
\draw [->] (3.35,-3.3) -- (4,-3.3);
\draw [->] (4.9,-3.3) -- (5.4,-3.3);

\draw [->] (0.2,-3.9) -- (0.8,-3.9);
\draw [->] (1.2, -3.9) -- (1.8,-3.9);
\draw [->] (2.2,-3.9) -- (2.8,-3.9);
\draw [->] (3.35,-3.9) -- (4,-3.9);
\draw [->] (4.9,-3.9) -- (5.4,-3.9);
\draw [->] (5.95, -3.9) -- (6.5,-3.9);

\draw [->] (0,-2.9) -- (0,-3.7);
\draw [->] (1, -2.9) -- (1,-3.1);
\draw [->] (2,-2.9) -- (2,-3.1);
\draw [->] (4.5,-2.9) -- (4.5,-3.1);
\draw [->] (5.7,-2.9) -- (5.7,-3.1);
\draw [->] (7.05,-2.9) -- (7.05,-3.7);

\draw [->] (1, -3.5) -- (1,-3.7);
\draw [->] (2,-3.5) -- (2,-3.7);
\draw [->] (4.5,-3.5) -- (4.5,-3.7);
\draw [->] (5.7,-3.5) -- (5.7,-3.7);
\end{tikzpicture}\]

The following module $M(\lambda)$ ($\lambda\not=0$) is indecomposable and satisfies (\ref{eq:homog}).

\[\begin{tikzpicture}[scale=1.3]
\node (-20) at (0,0) {$1$};
\node (-20) at (1,0) {$2$};
\node (-20) at (2,0) {$3$};
\node (-20) at (3.1,0) {$\cdots$};
\node (-20) at (4.5,0) {$n-1$};
\node (-20) at (5.7,0) {$n$};
\node (-20) at (7.05,0) {$n+1$};

\node (-20) at (0,-1.2) {$1$};
\node (-20) at (1,-0.6) {$2$};
\node (-20) at (2,-0.6) {$3$};
\node (-20) at (3.1,-0.6) {$\cdots$};
\node (-20) at (4.5,-0.6) {$n-1$};
\node (-20) at (5.7,-0.6) {$n$};
\node (-20) at (7.05,-1.2) {$n+1$};
\node (-20) at (6.4,-0.6) {$\lambda$};
\node (-20) at (6.4,-3.3) {$\lambda$};
\node (-20) at (6.4,-4.5) {$\lambda$};

\node (-20) at (1,-1.2) {$2$};
\node (-20) at (2,-1.2) {$3$};
\node (-20) at (3.1,-1.2) {$\cdots$};
\node (-20) at (4.5,-1.2) {$n-1$};
\node (-20) at (5.7,-1.2) {$n$};
\node (-20) at (7.05,-1.2) {$n+1$};

\draw [->] (0.2,0) -- (0.8,0);
\draw [->] (1.2, 0) -- (1.8,0);
\draw [->] (2.2,0) -- (2.8,0);
\draw [->] (3.35, 0) -- (4,0);
\draw [->] (4.9,0) -- (5.4,0);
\draw [->] (5.95, 0) -- (6.5,0);

\draw [->] (1.2, -0.6) -- (1.8,-0.6);
\draw [->] (2.2,-0.6) -- (2.8,-0.6);
\draw [->] (3.35,-0.6) -- (4,-0.6);
\draw [->] (4.9,-0.6) -- (5.4,-0.6);
\draw [->] (5.9,-0.6) -- (6.7,-0.2);

\draw [->] (0.2,-1.2) -- (0.8,-1.2);
\draw [->] (1.2, -1.2) -- (1.8,-1.2);
\draw [->] (2.2,-1.2) -- (2.8,-1.2);
\draw [->] (3.35,-1.2) -- (4,-1.2);
\draw [->] (4.9,-1.2) -- (5.4,-1.2);
\draw [->] (5.95, -1.2) -- (6.5,-1.2);

\draw [->] (0,-0.2) -- (0,-1);
\draw [->] (1, -0.2) -- (1,-0.4);
\draw [->] (2,-0.2) -- (2,-0.4);
\draw [->] (4.5,-0.2) -- (4.5,-0.4);
\draw [->] (5.7,-0.2) -- (5.7,-0.4);
\draw [->] (7.05,-0.2) -- (7.05,-1);

\draw [->] (1, -0.8) -- (1,-1);
\draw [->] (2,-0.8) -- (2,-1);
\draw [->] (4.5,-0.8) -- (4.5,-1);
\draw [->] (5.7,-0.8) -- (5.7,-1);

\draw [->] (0,-1.4) -- (0,-1.7);
\draw [->] (1, -1.4) -- (1,-1.7);
\draw [->] (2,-1.4) -- (2,-1.7);
\draw [->] (4.5,-1.4) -- (4.5,-1.7);
\draw [->] (5.7,-1.4) -- (5.7,-1.7);
\draw [->] (7.05,-1.4) -- (7.05,-1.7);

\node (-20) at (0,-1.85) {$\vdots$};
\node (-20) at (1,-1.85) {$\vdots$};
\node (-20) at (2,-1.85) {$\vdots$};
\node (-20) at (4.5,-1.85) {$\vdots$};
\node (-20) at (5.7,-1.85) {$\vdots$};
\node (-20) at (7.05,-1.85) {$\vdots$};

\draw [->] (0,-2.2) -- (0,-2.5);
\draw [->] (1, -2.2) -- (1,-2.5);
\draw [->] (2,-2.2) -- (2,-2.5);
\draw [->] (4.5,-2.2) -- (4.5,-2.5);
\draw [->] (5.7,-2.2) -- (5.7,-2.5);
\draw [->] (7.05,-2.2) -- (7.05,-2.5);

\node (-20) at (0,-2.7) {$1$};
\node (-20) at (1,-2.7) {$2$};
\node (-20) at (2,-2.7) {$3$};
\node (-20) at (3.1,-2.7) {$\cdots$};
\node (-20) at (4.5,-2.7) {$n-1$};
\node (-20) at (5.7,-2.7) {$n$};
\node (-20) at (7.05,-2.7) {$n+1$};

\node (-20) at (0,-3.9) {$1$};
\node (-20) at (1,-3.3) {$2$};
\node (-20) at (2,-3.3) {$3$};
\node (-20) at (3.1,-3.3) {$\cdots$};
\node (-20) at (4.5,-3.3) {$n-1$};
\node (-20) at (5.7,-3.3) {$n$};

\node (-20) at (1,-3.9) {$2$};
\node (-20) at (2,-3.9) {$3$};
\node (-20) at (3.1,-3.9) {$\cdots$};
\node (-20) at (4.5,-3.9) {$n-1$};
\node (-20) at (5.7,-3.9) {$n$};
\node (-20) at (7.05,-3.9) {$n+1$};

\draw [->] (0.2,-2.7) -- (0.8,-2.7);
\draw [->] (1.2,-2.7) -- (1.8,-2.7);
\draw [->] (2.2,-2.7) -- (2.8,-2.7);
\draw [->] (3.35, -2.7) -- (4,-2.7);
\draw [->] (4.9,-2.7) -- (5.4,-2.7);
\draw [->] (5.95, -2.7) -- (6.5,-2.7);

\draw [->] (1.2, -3.3) -- (1.8,-3.3);
\draw [->] (2.2,-3.3) -- (2.8,-3.3);
\draw [->] (3.35,-3.3) -- (4,-3.3);
\draw [->] (4.9,-3.3) -- (5.4,-3.3);
\draw [->] (5.9,-3.3) -- (6.7,-2.9);

\draw [->] (0.2,-3.9) -- (0.8,-3.9);
\draw [->] (1.2, -3.9) -- (1.8,-3.9);
\draw [->] (2.2,-3.9) -- (2.8,-3.9);
\draw [->] (3.35,-3.9) -- (4,-3.9);
\draw [->] (4.9,-3.9) -- (5.4,-3.9);
\draw [->] (5.95, -3.9) -- (6.5,-3.9);

\draw [->] (0,-2.9) -- (0,-3.7);
\draw [->] (1, -2.9) -- (1,-3.1);
\draw [->] (2,-2.9) -- (2,-3.1);
\draw [->] (4.5,-2.9) -- (4.5,-3.1);
\draw [->] (5.7,-2.9) -- (5.7,-3.1);
\draw [->] (7.05,-2.9) -- (7.05,-3.7);

\draw [->] (1, -3.5) -- (1,-3.7);
\draw [->] (2,-3.5) -- (2,-3.7);
\draw [->] (4.5,-3.5) -- (4.5,-3.7);
\draw [->] (5.7,-3.5) -- (5.7,-3.7);

\node (-20) at (1,-4.5) {$2$};
\node (-20) at (2,-4.5) {$3$};
\node (-20) at (3.1,-4.5) {$\cdots$};
\node (-20) at (4.5,-4.5) {$n-1$};
\node (-20) at (5.7,-4.5) {$n$};
\draw [->] (1, -4.1) -- (1,-4.3);
\draw [->] (2,-4.1) -- (2,-4.3);
\draw [->] (4.5,-4.1) -- (4.5,-4.3);
\draw [->] (5.7,-4.1) -- (5.7,-4.3);
\draw [->] (1.2, -4.5) -- (1.8,-4.5);
\draw [->] (2.2,-4.5) -- (2.8,-4.5);
\draw [->] (3.35,-4.5) -- (4,-4.5);
\draw [->] (4.9,-4.5) -- (5.4,-4.5);
\draw [->] (5.95, -4.5) -- (6.7,-4.1);
\end{tikzpicture}\]
Indeed, observe that $\End M(\lambda)\cong K$ and so $M$ is indecomposable. Moreover,
$M(\lambda)$ has the following projective resolution,
\[
\begin{tikzpicture}[scale=0.6,
 arr/.style={black, -angle 60}]
 \path (0, 1) node (c0) {$0$};
\path (3,1) node (c1) {$P_{n+1}$};
\path (7,1) node (c2) {$P_1$};
\path (11,1) node (c3) {$M(\lambda)$};
\path (14, 1) node (c4) {$0$};

\draw[->] (c1) edge node[auto]
{$\begin{smallmatrix}\lambda r_{\rho}-r_{\rho'}\end{smallmatrix}$} (c2);
\draw[->] (c2) edge node[auto]{}  (c3);
\draw[->] (c0) edge (c1);
\draw[->] (c3) edge (c4);
\end{tikzpicture}\]
where $\rho=a_{n}a_{n-1}\dots a_2a_1$ and $\rho'=a_{n}\epsilon_na_{n-1}\dots a_2a_1$.
By Proposition \ref{Prop:tau} and direct computation following (\ref{eqn:1}),

\[\tau M=\ker (\lambda l^*_{\rho}-l^*_{\rho'})=N(\lambda)\] is the module below.
$$\xymatrixcolsep{1.5pc}\xymatrixrowsep{1pc}\xymatrix{
 & 2\ar[r]\ar[d] & 3\ar[r]\ar[d] & \cdots \ar[r] & n-1\ar[r]\ar[d] & n\ar[d]\\
1\ar[r]\ar[dd]\ar[ru]^{\lambda}& 2 \ar[r]\ar[d] & 3\ar[r]\ar[d] & \cdots \ar[r] & n-1\ar[r]\ar[d] & n\ar[d]\ar[r]&n+1\ar[dd]\\
& 2\ar[r]\ar[d] & 3\ar[r]\ar[d] & \cdots \ar[r] & n-1\ar[r]\ar[d] & n\ar[d]\\
1\ar[r]\ar[d]\ar[ru]^{\lambda} & 2\ar[r]\ar[d]& 3\ar[r]\ar[d] & \cdots \ar[r] & n-1\ar[r]\ar[d] & n\ar[d] \ar[r]&n+1\ar[d]\\
\vdots\ar[dd] &\vdots\ar[d] & \vdots\ar[d] &  & \vdots \ar[d] & \vdots\ar[d] & \vdots\ar[dd]\\
& 2\ar[r]\ar[d] & 3\ar[r]\ar[d] & \cdots \ar[r] & n-1\ar[r]\ar[d] & n\ar[d]\\
1\ar[r]\ar[dd]\ar[ru]^{\lambda}& 2 \ar[r]\ar[d] & 3\ar[r]\ar[d] & \cdots \ar[r] & n-1\ar[r]\ar[d] & n\ar[d]\ar[r]&n+1\ar[dd]\\
 & 2\ar[r]\ar[d] & 3\ar[r]\ar[d] & \cdots \ar[r] & n-1\ar[r]\ar[d] & n\ar[d]\\
1\ar[r]\ar[ru]^{\lambda}&2\ar[r]&3\ar[r]&\cdots\ar[r]&n-1\ar[r]&n\ar[r] & n+1\\
}$$
By change of basis, we see that $M(\lambda)\cong N(\lambda)$, that is, $\tau M(\lambda)\cong M(\lambda)$.
By construction, $\rankv M=\delta$. Therefore
$M$ is a module as claimed.

\subsection{Type $\type{G}_{21}$}\label{sec:homogG21}
 In this case, $C=\left(\begin{array}{ccc}
                           2 & -1 & 0 \\
                           -1 & 2 & -3 \\
                           0 & -1 & 2 \\
                         \end{array}
                       \right)
$, $D=\text{diag}(m, m, 3m)$. Let $\Omega=\{(2,1),(3,2)\}$. Then  $H=H(C,D,\Omega)$ is given by the quiver \[\begin{tikzpicture}[scale=1.1]
\node (-20) at (0,0) {$1$};
\node (-20) at (1.5,0) {$2$};
\node (-20) at (3,0) {$3$};

\node (-20) at (0.75,-0.2) {$a_{1}$};
\node (-20) at (2.25,-0.2) {$a_{2}$};

\draw [->] (0.25,0) -- (1.25,0);
\draw [->] (1.75, 0) -- (2.75,0);

\draw[-latex] (-0.2,0.1) .. controls (-0.5,0.7) and (0.5,0.7) .. (0.2,0.05);
\draw[-latex] (1.3,0.1) .. controls (1,0.7) and (2,0.7) .. (1.7,0.05);
\draw[-latex] (2.8,0.1) .. controls (2.5,0.7) and (3.5,0.7) .. (3.2,0.05);

\node (-20) at (1.5,0.8) {$\varepsilon_2$};
\node (-20) at (0.0,0.8) {$\varepsilon_1$};
\node (-20) at (3.0,0.8) {$\varepsilon_3$};
\end{tikzpicture}\]
with relation $\varepsilon_i^m=0$ for $i=1,2$, $\varepsilon_3^{3m}=0$, $a_1\varepsilon_1=\varepsilon_2a_1$ and $a_2\varepsilon_2=\varepsilon_3^{3}a_2$. We have $\delta=(1, 2, 1)$.

The following module $M$ is indecomposable and satisfies (\ref{eq:homog}).
$$\xymatrixcolsep{1.5pc}\xymatrixrowsep{1pc}\xymatrix{
1\ar[rrr]\ar[dd]&&& 1\ar[rr]\ar[dd]&&\cdots\ar[r]& 1\ar[rrr]\ar[dd]&&&1\ar[dd]\\
&2\ar[rrr]\ar[dd] &&& 2\ar[r]\ar[dd] & \cdots \ar[rr] &&2\ar[rrr]\ar[dd]&&&2\ar[dd]\\
2\ar[rrr]\ar[d] &&&2\ar[rr]\ar[d] & &\cdots\ar[r] & 2\ar[rrr]\ar[d]&&&2\ar[d]\\
3\ar[r] & 3\ar[r] & 3\ar[r] & 3\ar[r] & 3\ar[r] & \cdots\ar[r] & 3 \ar[r] & 3\ar[r] & 3\ar[r] &3\ar[r] & 3\ar[r] & 3 \\
}$$
Indeed, again by observation, we have $ \End M\cong K$ and so $M$ is indecomposable. Moreover,
$M$ has the following projective resolution,
\[
\begin{tikzpicture}[scale=0.6,
 arr/.style={black, -angle 60}]
 \path (0, 1) node (c0) {$0$};
\path (3,1) node (c1) {$P_{3}$};
\path (7,1) node (c2) {$P_1\oplus P_2$};
\path (11,1) node (c3) {$M$};
\path (14, 1) node (c4) {$0$.};

\draw[->] (c1) edge node[auto]
{$\left(\begin{smallmatrix}r_{\varepsilon_3a_2a_1}\\-r_{a_2}\end{smallmatrix}\right)$} (c2);
\draw[->] (c2) edge node[auto]{}  (c3);
\draw[->] (c0) edge (c1);
\draw[->] (c3) edge (c4);
\end{tikzpicture}\]
By Proposition \ref{Prop:tau} and direct computation following (\ref{eqn:1}), we have
\[\tau M=\ker \left(\begin{smallmatrix}l^*_{\varepsilon_3a_2a_1}\\-l^*_{a_2}\end{smallmatrix}\right)\cong M.\]
Therefore, $M$ is a module as claimed.

\section{Inhomogeneous tubes with rigid mouth modules}\label{sec4}
In this section, let $C$ be a non-simply laced Cartan matrix of affine type,
$D$  a minimal symmetriser of $C$,
 and  $\Omega$ the same orientation as that of $\val$ in \cite[Section 6]{[DR]}.
Recall that a tube is called good if the mouth modules are rigid. 
For each inhomogeneous tube $\mathcal{C}$ in the AR-quiver of $\val$, we construct a
good tube $\mathcal{T}$ of $\tau$-locally free $H$-modules of the same rank such that 
the rank vectors of the mouth modules are the same as the dimension vectors 
of the mouth modules of $\mathcal{C}$.

\subsection{Type $\type{B}_n$} In this case, we refer the reader to \S \ref{sec:homogBn}
for the definition of $H$ and examples of indecomposable projective and injective
$H$-modules, where $m=1$. The following proposition is a finer version of Lemma \ref{lem0} for this case.
\begin{prop}\label{typeB}
The AR-quiver $\Gamma_H$ has a good tube 
of rank $n$ with the mouth modules,
$$\xymatrix{
 M_{\type B} & E_n  & \cdots & E_3 & E_2,
}$$
where
$$\xymatrix{
&1\ar[r] &2\ar[r]\ar[d] & 3 \ar[r]\ar[d] & \cdots\ar[r] &n-1 \ar[r]\ar[d] & n \ar[r]\ar[d]& n+1\\
M_{\type B}: & 1\ar[r] &2\ar[r] & 3\ar[r] &\cdots\ar[r] & n-1\ar[r] & n\ar[r] & n+1.}$$
\end{prop}
\begin{proof}
Observe from the list of type $\type{B}_n$ in \cite[Section 6]{[DR]}, we have $$c\alpha_2=\alpha_3,\cdots,c^{n-2}\alpha_2=\alpha_{n},c^{n-1}\alpha_2=(2,1,1,\cdots,1,1,2), c^n\alpha_2=\alpha_2.$$
Therefore,
\[
\tau^i E_2\cong E_{2+i} \text{ for any } 1\leq i \leq n-2, \text{ and }\ \tau^nE_2\cong E_2.
\]
We compute $\tau E_n$. $E_n$ has the following minimal projective resolution,
\[
\begin{tikzpicture}[scale=0.6,
 arr/.style={black, -angle 60}]
 \path (-3, 1) node (c0) {$0$};
\path (1,1) node (c1) {$P_{n+1}\oplus P_{n+1}$};
\path (7,1) node (c2) {$P_n$};
\path (11,1) node (c3) {$E_n$};
\path (14, 1) node (c4) {$0$.};

\draw[->] (c1) edge node[auto]
{$\left(\begin{smallmatrix}r_{a_n }& -r_{a_n\varepsilon_n}\end{smallmatrix}\right)$} (c2);
\draw[->] (c2) edge node[auto]{}  (c3);
\draw[->] (c0) edge (c1);
\draw[->] (c3) edge (c4);
\end{tikzpicture}\]
 So by Proposition \ref{Prop:tau}, $M_{\type B}=
 \ker \left(\begin{matrix}l^*_{a_n} & -l^*_{a_n\varepsilon_n}\end{matrix}\right)$, which is a module as claimed. The modules
 are all rigid by Lemma \ref{lem0}.
\end{proof}

\subsection{Type $\type C_n$}
In this case,
$$C=\left(
                \begin{array}{ccccccc}
                  2 & -1 &  &  & & \\
                  -2 & 2 & -1 &  &  &\\
                  & -1 & 2 & -1 & & \\
                 & & \ddots & \ddots & \ddots & \\
                 & & &  & -1 & 2 & -2\\
                 & &  &  &  & -1 & 2 \\
                \end{array} \right),$$
 $D=\textup{diag}(2, ~1, ~1, \cdots,~1, ~1, ~2)$ and $\Omega=\{(2,1),(3,2),\cdots,(n+1,n)\}$.
Then $H=H(C,D,\Omega)$  is given by the quiver \[\begin{tikzpicture}[scale=1.1]
\node (-20) at (0,0) {$1$};
\node (-20) at (1.5,0) {$2$};
\node (-20) at (3,0) {$\cdots$};
\node (-20) at (4.5,0) {$n$};
\node (-20) at (6.3,-0.02) {$n+1$};

\node (-20) at (0.75,-0.2) {$a_{1}$};
\node (-20) at (2.25,-0.2) {$a_{2}$};
\node (-20) at (3.75,-0.2) {$a_{n-1}$};
\node (-20) at (5.2,-0.2) {$a_{n}$};

\draw [->] (0.25,0) -- (1.25,0);
\draw [->] (1.75, 0) -- (2.75,0);
\draw [->] (3.25,0) -- (4.25,0);
\draw [->] (4.75, 0) -- (5.8,0);

\draw[-latex] (-0.2, 0.1) .. controls (-0.5,0.7) and (0.5,0.7) .. (0.2,0.05);
\draw[-latex] (6.1,0.1) .. controls (5.8,0.7) and (6.8,0.7) .. (6.5,0.07);

\node (-20) at (0,0.8) {$\varepsilon_1$};
\node (-20) at (6.3,0.8) {$\varepsilon_{n+1}$};
\end{tikzpicture}\]
with relations
$\varepsilon_1^{2}=\varepsilon_{n+1}^{2}=0$.

\begin{prop}\cite{[HLS]}\label{typeC}
The AR-quiver $\Gamma_H$ has a good tube of rank $n$  
with the  mouth modules,
$$\xymatrix{M_{\type C}&  E_n &  \cdots &E_3& E_2},
$$
where
$$
\xymatrixcolsep{1.5pc}\xymatrixrowsep{1pc}\xymatrix{
&& 1\ar[d]&&&&&& n+1\\
&M_{\type C}\colon &1\ar[r]&2\ar[r]&3\ar[r]&\cdots\ar[r]&n-1\ar[r]&n \ar[r]&n+1. \ar[u] \\
}$$
\end{prop}

We remark that in this case $H$ is a string algebra and AR-sequences
are explicitly described in  \cite{[HLS]}.

\subsection{Type $\type{BC}_n$} In this case,
 $$C=\left(
                         \begin{array}{cccccccc}
                           2 & -2 & 0 & 0 & \cdots & 0 & 0 & 0 \\
                           -1 & 2 & -1 & 0 & \cdots &0 &0 & 0 \\
                           0 & -1 & 2 & -1& \cdots &0 & 0& 0\\
                           \vdots & \vdots &\vdots &\vdots &  & \vdots & \vdots & \vdots\\
                            0 & 0 &0 &0 &\cdots &2 &-1 &0 \\
                            0 & 0 &0 & 0&\cdots &-1 &2 & -2\\
                           0 & 0 & 0 &0 &\cdots& 0 &-1 & 2\\
                         \end{array}
                       \right),$$
 $D=\text{diag}(1,2,2,\cdots,2,2,4)$  and $\Omega=\{(2,1),(3,2),(4,3),\cdots, (n+1,n)\}.$
Then  $H=H(C,D,\Omega)$ is given by the quiver \[\begin{tikzpicture}[scale=1.1]
\node (-20) at (-1.5,0.0) {$1$};
\node (-20) at (0,0) {$2$};
\node (-20) at (1.5,0) {$3$};
\node (-20) at (3,0) {$\cdots$};
\node (-20) at (4.5,0) {$n$};
\node (-20) at (6.3,0) {$n+1$};

\node (-20) at (-0.75,-0.2) {$a_{1}$};
\node (-20) at (0.75,-0.2) {$a_{2}$};
\node (-20) at (2.25,-0.2) {$a_{3}$};
\node (-20) at (3.75,-0.2) {$a_{n-1}$};
\node (-20) at (5.25,-0.2) {$a_{n}$};

\draw [->] (-1.25,0) -- (-0.25,0);
\draw [->] (0.25,0) -- (1.25,0);
\draw [->] (1.75, 0) -- (2.75,0);
\draw [->] (3.25,0) -- (4.25,0);
\draw [->] (4.75, 0) -- (5.8,0);

\draw[-latex] (-0.2,0.1) .. controls (-0.5,0.7) and (0.5,0.7) .. (0.2,0.05);
\draw[-latex] (1.3,0.1) .. controls (1,0.7) and (2,0.7) .. (1.7,0.05);
\draw[-latex] (4.3,0.1) .. controls (4,0.7) and (5,0.7) .. (4.7,0.05);
\draw[-latex] (6.1,0.1) .. controls (5.8,0.7) and (6.8,0.7) .. (6.5,0.05);

\node (-20) at (1.5,0.8) {$\varepsilon_3$};
\node (-20) at (0,0.8) {$\varepsilon_{2}$};
\node (-20) at (4.5,0.8) {$\varepsilon_n$};
\node (-20) at (6.3,0.8) {$\varepsilon_{n+1}$};
\end{tikzpicture}\]
with relations
$$\varepsilon_{i}^2=0,~ \varepsilon_{n+1}^4=0, ~\varepsilon_{i+1}a_{i}=a_{i}\varepsilon_i, ~\varepsilon_{n+1}^2a_{n}=a_{n}\varepsilon_n$$
for $2\leq i\leq n$.

\begin{prop}\label{typeBC}
The AR-quiver  $\Gamma_H$ has a good  tube 
 of rank $n$ with the mouth modules
$$\xymatrix{  M_{\type {BC}} & E_n &  \cdots &E_3 & E_2,}
$$
where
$$
\xymatrixcolsep{1.5pc}\xymatrixrowsep{1pc}\xymatrix{
  & 1\ar[r]&2\ar[dd]\ar[r]&3\ar[r]\ar[dd]&\cdots\ar[r]&n-1\ar[r]\ar[dd]&n \ar[r]\ar[dd]&n+1\ar[d]\\
 &&&&&&&n+1\ar[d]\\
M_{\type{BC}}\colon &1\ar[r]&2\ar[r]&3\ar[r]&\cdots\ar[r]&n-1\ar[r]&n\ar[r]&n+1\ar[d]\\
&&&&&&&n+1.\\
}$$
\end{prop}

As the proof of Proposition \ref{typeB} (see also Lemma \ref{lem0}), the key step of the proof is to show
that $\tau E_n=M_{\type{BC}}$, which can be computed following Proposition \ref{Prop:tau} and
is indeed as claimed above. We skip the details of the computation. Similarly, we have
Propositions \ref{typeBD1} and \ref{typeCD1}.

\subsection{Type $\type{BD}_n$}
In this case,  $$C=\left(
                         \begin{array}{cccccccc}
                           2 & 0 & -1 & 0 & \cdots & 0 & 0 & 0 \\
                           0 & 2 & -1 & 0 & \cdots &0 &0 & 0 \\
                           -1 & -1 & 2 & -1& \cdots &0 & 0& 0\\
                           \vdots & \vdots &\vdots &\vdots &  & \vdots & \vdots & \vdots\\
                            0 & 0 &0 &0 &\cdots &2 &-1 &0 \\
                            0 & 0 &0 & 0&\cdots &-1 &2 & -1\\
                           0 & 0 & 0 &0 &\cdots& 0 &-2 & 2\\
                         \end{array}
                       \right),$$
 $D=\text{diag}(2,2,2,\cdots,2,2,1)$ and $\Omega=\{(3,1),(3,2),(4,3),\cdots, (n+1,n)\}.$
Then  $H=H(C,D,\Omega)$ is given by the quiver

\[\begin{tikzpicture}[scale=1.1]
\node (-20) at (-1.5,0.75) {$1$};
\node (-20) at (-1.5,-0.75) {$2$};
\node (-20) at (0,0) {$3$};
\node (-20) at (1.5,0) {$4$};
\node (-20) at (3,0) {$\cdots$};
\node (-20) at (4.5,0) {$n$};
\node (-20) at (6.3,0) {$n+1$};

\node (-20) at (-0.8,0.2) {$a_{1}$};
\node (-20) at (-0.7,-0.7) {$a_{2}$};
\node (-20) at (0.75,-0.2) {$a_{3}$};
\node (-20) at (2.25,-0.2) {$a_{4}$};
\node (-20) at (3.75,-0.2) {$a_{n-1}$};
\node (-20) at (5.35,-0.2) {$a_{n}$};

\draw [->] (-1.3,0.75) -- (-0.3,0.2);
\draw [->] (-1.3,-0.75) -- (-0.2,-0.2);
\draw [->] (0.25,0) -- (1.25,0);
\draw [->] (1.75, 0) -- (2.75,0);
\draw [->] (3.25,0) -- (4.25,0);
\draw [->] (4.75, 0) -- (5.75,0);

\draw[-latex] (-1.7,0.85) .. controls (-2,1.45) and (-1,1.45) .. (-1.3,0.8);
\draw[-latex] (-1.7,-0.65) .. controls (-2,-0.05) and (-1,-0.05) .. (-1.3,-0.7);
\draw[-latex] (-0.2,0.1) .. controls (-0.5,0.7) and (0.5,0.7) .. (0.2,0.05);
\draw[-latex] (1.3,0.1) .. controls (1,0.7) and (2,0.7) .. (1.7,0.05);
\draw[-latex] (4.3,0.1) .. controls (4,0.7) and (5,0.7) .. (4.7,0.05);

\node (-20) at (1.5,0.8) {$\varepsilon_4$};
\node (-20) at (0,0.8) {$\varepsilon_{3}$};
\node (-20) at (-1.5,1.55) {$\varepsilon_{1}$};
\node (-20) at (-1.5,0.0) {$\varepsilon_{2}$};
\node (-20) at (4.5,0.8) {$\varepsilon_n$};
\end{tikzpicture}\]
with relation $\varepsilon_{i}^2=0$ for $1\leq i\leq n$, $\varepsilon_3a_{1}=a_{1}\varepsilon_1$ and $\varepsilon_{i+1}a_{i}=a_{i}\varepsilon_i$ for $2\leq i\leq n-1$.

\begin{prop}\label{typeBD1}
The AR-quiver  $\Gamma_H$ has a good tube 
of rank $n-1$ with the mouth modules
$$ \xymatrix{M^1_{\type {BD}} &  E_n&  \cdots &E_4  &E_3,}
$$
where
$$
\xymatrixcolsep{1.5pc}\xymatrixrowsep{1pc}\xymatrix{
 M^1_{\type{BD}}\colon &1\ar[rd]\ar[d]&&&&&&\\
&1\ar[rd]&3\ar[d]\ar[r]&4\ar[r]\ar[d]&\cdots\ar[r]&n\ar[r]\ar[d]&n+1\\
&2\ar[ru]\ar[d]&3\ar[r]&4\ar[r]&\cdots\ar[r]&n\ar[r]&n+1.\\
&2\ar[ru]&&&&&\\
}$$ 
\end{prop}

\begin{prop}\label{typeBD2}
The AR-quiver  $\Gamma_H$ has a good tube 
 of rank $2$ with the two mouth modules,
$$\xymatrixcolsep{1.5pc}\xymatrixrowsep{1pc}\xymatrix{
M^2_{\type {BD}}\colon&1\ar[r]\ar[d] &3\ar[r]\ar[d] &4\ar[r]\ar[d] & \cdots\ar[r]&n-1\ar[r]\ar[d] &n\ar[d] & \\
&1\ar[r] &3\ar[r] &4\ar[r] & \cdots\ar[r]&n-1\ar[r] &n\ar[r] & n+1\\} $$

and

$$\xymatrixcolsep{1.5pc}\xymatrixrowsep{1pc}
\xymatrix{ M^3_{\type {BD}}\colon &
2\ar[r]\ar[d] &3\ar[r]\ar[d] &4\ar[r]\ar[d] & \cdots\ar[r]&n-1\ar[r]\ar[d] &n\ar[d] & \\
&2\ar[r] &3\ar[r] &4\ar[r] & \cdots\ar[r]&n-1\ar[r] &n\ar[r] & n+1.\\}$$
\end{prop}
\begin{proof}
By observation, $\dim \End M^2_{\type {BD}}=1$ and $\rankv M^2_{\type {BD}}=(1, 0, 1, \dots, 1)$.
Compute the bilinear form,
\[
\langle\rankv M^2_{\type {BD}}, \rankv M^2_{\type {BD}} \rangle
=n\cdot 2 \cdot 1\cdot 1 +1-(n-1)  2 \cdot 1\cdot 1-2=1.
\]
Therefore, by Proposition \ref{lem2.2}, $M^2_{\type {BD}}$ is rigid. Similarly, $M^3_{\type {BD}}$
is rigid as well.
Therefore,  by Proposition \ref{lem2.1},  the two modules are $\tau$-locally free.
Note that (cf \cite[Section 6]{[DR]}),
\[c(\rankv M^2_{\type {BD}})=\rankv M^3_{\type {BD}} \text{ and } c(\rankv M^3_{\type {BD}})=\rankv M^2_{\type {BD}},\]
and so  by Theorem \ref{thm:Schurroot},
\[\tau M^2_{\type {BD}}\cong M^3_{\type {BD}} \text{ and } \tau M^3_{\type {BD}}\cong M^2_{\type {BD}},\]
that is, the two modules form a $\tau$-orbit in a  tube of rank 2.
As $(M^2_{\type {BD}})_2=0$ and  $(M^3_{\type {BD}})_1=0$,
by similar arguments as the part of proof in Lemma \ref{lem0.2} that the module $M'$
is a mouth module, the two modules are mouth modules. This completes the proof.
\end{proof}

\subsection{Type $\type{CD}_n$}\label{typeCD}
In this case,  $$C=\left(
                         \begin{array}{cccccccc}
                           2 & 0 & -1 & 0 & \cdots & 0 & 0 & 0 \\
                           0 & 2 & -1 & 0 & \cdots &0 &0 & 0 \\
                           -1 & -1 & 2 & -1& \cdots &0 & 0& 0\\
                           \vdots & \vdots &\vdots &\vdots & & \vdots & \vdots & \vdots\\
                            0 & 0 &0 &0 &\cdots &2 &-1 &0 \\
                            0 & 0 &0 & 0&\cdots &-1 &2 & -2\\
                           0 & 0 & 0 &0 &\cdots& 0 &-1 & 2\\
                         \end{array}
                       \right),$$
$D=\text{diag}(1,\,1,\,1,\, \cdots,\,1,\,1,\,2)$ and $\Omega=\{(3,1), ~ (3,2),~ (4,3),  ~ \cdots, ~(n+1,n)\}.$
Then  $H=H(C,D,\Omega)$ is given by the quiver
\[\begin{tikzpicture}[scale=1.2]
\node (-20) at (-1,0.5) {$1$};
\node (-20) at (-1,-0.5) {$2$};
\node (-20) at (0,0) {$3$};
\node (-20) at (1.5,0) {$4$};
\node (-20) at (3,0) {$\cdots$};
\node (-20) at (4.5,0) {$n$};
\node (-20) at (6.3,0) {$n+1$};

\node (-20) at (-0.4,0.5) {$a_{1}$};
\node (-20) at (-0.4,-0.5) {$a_{2}$};
\node (-20) at (0.75,-0.2) {$a_{3}$};
\node (-20) at (2.25,-0.2) {$a_{4}$};
\node (-20) at (3.75,-0.2) {$a_{n-1}$};
\node (-20) at (5.3,-0.2) {$a_{n}$};

\draw [->] (-0.85,0.5) -- (-0.15,0.1);
\draw [->] (-0.88,-0.5) -- (-0.15,-0.1);
\draw [->] (0.25,0) -- (1.25,0);
\draw [->] (1.75, 0) -- (2.75,0);
\draw [->] (3.25,0) -- (4.25,0);
\draw [->] (4.75, 0) -- (5.75,0);

\draw[-latex] (6.1,0.1) .. controls (5.8,0.7) and (6.8,0.7) .. (6.5,0.05);

\node (-20) at (6.3,0.8) {$\varepsilon_{n+1}$};
\end{tikzpicture}\]
with relation $\varepsilon_{n+1}^2$.

\begin{prop}\label{typeCD1}
The AR-quiver  $\Gamma_H$ has a good tube 
 of rank $n-1$ with the mouth modules,
$$
\xymatrixcolsep{1pc}\xymatrixrowsep{1pc}
\xymatrix{
1\ar[rd]&&&&&&&&&&\\
&3\ar[r]&4\ar[r]&\cdots\ar[r]&n\ar[r]&n+1\ar[d]& E_n &\cdots& E_4  & E_3. \\
2\ar[ru]&&&&&n+1&&&&.\\
}$$
\end{prop}

\begin{prop}\label{typeCD2}
The AR-quiver $\Gamma_H$ has a good tube 
of rank $2$ with the two mouth modules,
$$\xymatrixcolsep{1.5pc}\xymatrixrowsep{1pc}\xymatrix{
2\ar[r] &3\ar[r] &4\ar[r] & \cdots\ar[r] &n\ar[r] & n+1\ar[d]\\
2\ar[r] &3\ar[r] &4\ar[r] & \cdots\ar[r] &n\ar[r] & n+1\\} \ \ \
\xymatrix{
1\ar[r] &3\ar[r] &4\ar[r] & \cdots\ar[r] &n\ar[r] & n+1\ar[d]\\
1\ar[r] &3\ar[r] &4\ar[r] & \cdots\ar[r] &n\ar[r] & n+1.\\}$$  
\end{prop}
Observe that the endomorphism rings of both modules are 2-dimensional and so they are
rigid, by Proposition \ref{lem2.2}. Now the proposition follows from similar arguments as in the proof of Proposition \ref{typeBD2}. We skip the details. Similarly, we have Propositions \ref{typeF1} to \ref{typeG2}. We skip the detailed proofs, but note only the
dimensions of the endomorphism rings of the modules constructed in each case.

\subsection{Type $\type{F}_{41}$}\label{type F1}
 In this case, $$C=\left(
                         \begin{array}{ccccc}
                           2 & -1 & 0 & 0 & 0 \\
                           -1 & 2 & -1 & 0 & 0 \\
                           0 & -1 & 2 &-2 & 0\\
                            0 & 0 &-1 &2 & -1\\
                           0 & 0 &0 &-1 & 2\\
                         \end{array}
                       \right),$$
 $D=\text{diag}(1,1,1,2,2)$ and $\Omega=\{(2,1),(3,2),(4,3),(4,5)\}$. Then  $H=H(C,D,\Omega)$ is given by the quiver \[\begin{tikzpicture}[scale=1.1]
\node (-20) at (0,0) {$1$};
\node (-20) at (1.5,0) {$2$};
\node (-20) at (3,0) {$3$};
\node (-20) at (4.5,0) {$4$};
\node (-20) at (6,0) {$5$};

\node (-20) at (0.75,-0.2) {$a_{1}$};
\node (-20) at (2.25,-0.2) {$a_{2}$};
\node (-20) at (3.75,-0.2) {$a_{3}$};
\node (-20) at (5.25,-0.2) {$a_{4}$};

\draw [->] (0.25,0) -- (1.25,0);
\draw [->] (1.75, 0) -- (2.75,0);
\draw [->] (3.25,0) -- (4.25,0);
\draw [->] (5.75, 0) -- (4.75,0);

\draw[-latex] (5.8,0.1) .. controls (5.5,0.7) and (6.5,0.7) .. (6.2,0.05);
\draw[-latex] (4.3,0.1) .. controls (4,0.7) and (5,0.7) .. (4.7,0.05);

\node (-20) at (6,0.8) {$\varepsilon_5$};
\node (-20) at (4.5,0.8) {$\varepsilon_4$};
\end{tikzpicture}\]
with relations $\varepsilon_i^2=0$ for $i=4,5$ and $\varepsilon_4a_{4}=a_{4}\varepsilon_5$.

\begin{prop}\label{typeF1}
The AR-quiver $\Gamma_H$ has two good tubes 
 of rank $2$, $3$, respectively.
\begin{itemize}
\item[(1)] The mouth modules of the rank 2 tube are
$$\xymatrixcolsep{1.5pc}\xymatrixrowsep{1pc}\xymatrix{
1\ar[r]&2\ar[r]&3\ar[r]&4\ar[d]\\
&&3\ar[r]&4\\
} \ \ \ \xymatrix{
2\ar[r]&3\ar[r]& 4\ar[d]&5\ar[l]\ar[d]\\
&&4&5.\ar[l]\\
}$$

\item[(2)] The mouth modules of the rank 3 tube are

$$\xymatrixcolsep{1.5pc}\xymatrixrowsep{1pc}
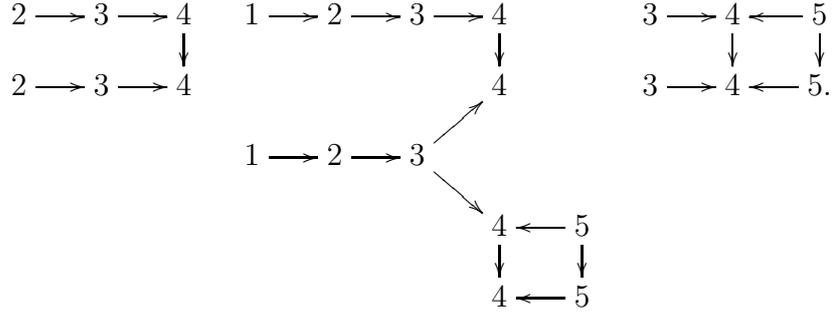
\xymatrix{
2\ar[r]& 3\ar[r]&4\ar[d]\\
2\ar[r]&3\ar[r]&4\\
} \ \ \
 \xymatrix{
1\ar[r]&2\ar[r]&3\ar[r]&4\ar[d]&\\
&&&4&\\
1\ar[r]&2\ar[r]&3\ar[ru]\ar[rd]&\\
&&&4\ar[d]&5\ar[l]\ar[d]\\
&&&4&5\ar[l]\\
} \ \ \ \xymatrix{
3\ar[r]& 4\ar[d]&5\ar[l]\ar[d]\\
3\ar[r]&4&5.\ar[l]\\
}$$
\end{itemize}
\end{prop}
Observe that the endomorphism rings of the  two modules in (1) are 1-dimensional and those of the
three modules in (2) are 2-dimensional. 

\subsection{Type $\type{F}_{42}$}\label{typeF2}
 In this case, $$C=\left(
                         \begin{array}{ccccc}
                           2 & -1 & 0 & 0 & 0 \\
                           -1 & 2 & -1 & 0 & 0 \\
                           0 & -1 & 2 &-1 & 0\\
                            0 & 0 &-2 &2 & -1\\
                           0 & 0 &0 &-1 & 2\\
                         \end{array}
                       \right),
$$ $D=\text{diag}(2,2,2,1,1)$ and $\Omega=\{(2,1),(3,2),(3,4),(4,5)\}$. Then  $H=H(C,D,\Omega)$ is given by the quiver 
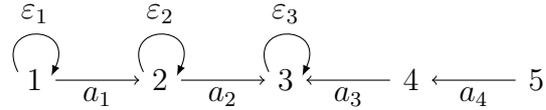
\[\begin{tikzpicture}[scale=1.1]
\node (-20) at (0,0) {$1$};
\node (-20) at (1.5,0) {$2$};
\node (-20) at (3,0) {$3$};
\node (-20) at (4.5,0) {$4$};
\node (-20) at (6,0) {$5$};

\node (-20) at (0.75,-0.2) {$a_{1}$};
\node (-20) at (2.25,-0.2) {$a_{2}$};
\node (-20) at (3.75,-0.2) {$a_{3}$};
\node (-20) at (5.25,-0.2) {$a_{4}$};

\draw [->] (0.25,0) -- (1.25,0);
\draw [->] (1.75, 0) -- (2.75,0);
\draw [->] (4.25,0) -- (3.25,0);
\draw [->] (5.75, 0) -- (4.75,0);

\draw[-latex] (-0.2,0.1) .. controls (-0.5,0.7) and (0.5,0.7) .. (0.2,0.05);
\draw[-latex] (1.3,0.1) .. controls (1,0.7) and (2,0.7) .. (1.7,0.05);
\draw[-latex] (2.8,0.1) .. controls (2.5,0.7) and (3.5,0.7) .. (3.2,0.05);

\node (-20) at (0,0.8) {$\varepsilon_1$};
\node (-20) at (1.5,0.8) {$\varepsilon_2$};
\node (-20) at (3,0.8) {$\varepsilon_3$};
\end{tikzpicture}\]
with relation $\varepsilon_i^2=0$ for $1\leq i\leq 3$ and $\varepsilon_{i+1}a_{i}=a_{i}\varepsilon_i$ for $i=1,2$.

\begin{prop}\label{typeF22}
The AR-quiver $\Gamma_H$ has two good tubes of rank $2$, $3$, respectively.
\begin{itemize}
\item[(1)]  The mouth modules of the rank 2 tube 
 are
$$\xymatrixcolsep{1.5pc}\xymatrixrowsep{1pc}
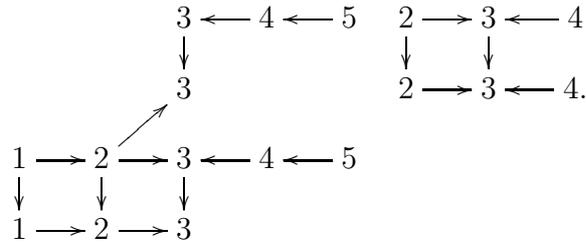
\xymatrix{
&&3\ar[d]&4\ar[l]&5\ar[l]\\
&&3&\\
1\ar[r]\ar[d]&2\ar[r]\ar[d]\ar[ru]&3\ar[d]&4\ar[l]&5\ar[l]\\
1\ar[r]&2\ar[r]&3\\
} \ \
\xymatrix{
2\ar[r]\ar[d]&3\ar[d]& 4\ar[l]\\
2\ar[r]&3&4.\ar[l]\\
}$$

\item[(2)]  The mouth modules of the rank 3 tube 
 are
$$\xymatrixcolsep{01.5pc}\xymatrixrowsep{1pc}
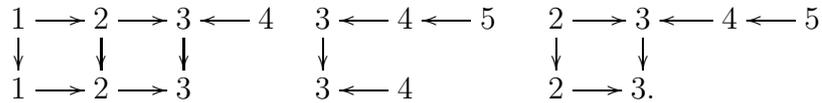
\xymatrix{
1\ar[r]\ar[d]&2\ar[r]\ar[d]&3\ar[d]& 4\ar[l]\\
1\ar[r]&2\ar[r]&3
} \ \  \xymatrix{
3\ar[d]& 4\ar[l]& 5\ar[l]\\
3&4\ar[l]&\\
} \ \ \ \xymatrix{
2\ar[r]\ar[d]&3\ar[d]& 4\ar[l]&5\ar[l]\\
2\ar[r]&3.&&\\&\\
}$$
\end{itemize}
\end{prop}

The endomorphism rings of the  two modules in (1) are 2-dimensional and those for the
three modules in (2) are 1-dimensional.

\subsection{Type $\type{G}_{21}$} \label{type g1}
 In this case, $C=\left(
                         \begin{array}{ccc}
                           2 & -1 & 0 \\
                           -1 & 2 & -3 \\
                           0 & -1 & 2 \\
                         \end{array}
                       \right)
$,  $D=\text{diag}(1,1,3)$ and $\Omega=\{(2,1),(3,2)\}$. Then  $H=H(C,D,\Omega)$ is given by the quiver \[\begin{tikzpicture}[scale=1.1]
\node (-20) at (0,0) {$1$};
\node (-20) at (1.5,0) {$2$};
\node (-20) at (3,0) {$3$};

\node (-20) at (0.75,-0.2) {$a_{1}$};
\node (-20) at (2.25,-0.2) {$a_{2}$};

\draw [->] (0.25,0) -- (1.25,0);
\draw [->] (1.75, 0) -- (2.75,0);

\draw[-latex] (2.8,0.1) .. controls (2.5,0.7) and (3.5,0.7) .. (3.2,0.05);

\node (-20) at (3.0,0.8) {$\varepsilon_3$};
\end{tikzpicture}\]
with relation $\varepsilon_3^3=0$.

\begin{prop}\label{typeG1}
The AR-quiver  $\Gamma_H$ has a good  tube of rank $2$ with the mouth modules
$$\xymatrixcolsep{1.5pc}\xymatrixrowsep{1pc}\xymatrix{
1\ar[r]&2\ar[r]&3\ar[d]\\
&&3\ar[d]\\
1\ar[r]&2\ar[ru]\ar[rd]&3\\
&&3\ar[d]\\
1\ar[r]&2\ar[r]&3\ar[d]\\
&&3\\
} \ \  \xymatrix{
2\ar[r]&3\ar[d]\\
2\ar[r]&3\ar[d]\\
2\ar[r]&3.\\
}$$
\end{prop}

The endomorphism rings of the  two modules are 3-dimensional.

\subsection{Type $\type{G}_{22}$}\label{typeG2}
 In this case,
 $C=\left(
                         \begin{array}{ccc}
                           2 & -1 & 0 \\
                           -1 & 2 & -1 \\
                           0 & -3 & 2 \\
                         \end{array}
                       \right)
$, $D=\text{diag}(3,3,1)$ and $\Omega=\{(2,1),(2,3)\}$. Then  $H=H(C,D,\Omega)$ is given by the quiver \[\begin{tikzpicture}[scale=1.1]
\node (-20) at (0,0) {$1$};
\node (-20) at (1.5,0) {$2$};
\node (-20) at (3,0) {$3$};

\node (-20) at (0.75,-0.2) {$a_{1}$};
\node (-20) at (2.25,-0.2) {$a_{2}$};

\draw [->] (0.25,0) -- (1.25,0);
\draw [->] (2.75, 0) -- (1.75,0);

\draw[-latex] (-0.2,0.1) .. controls (-0.5,0.7) and (0.5,0.7) .. (0.2,0.05);
\draw[-latex] (1.3,0.1) .. controls (1,0.7) and (2,0.7) .. (1.7,0.05);

\node (-20) at (1.5,0.8) {$\varepsilon_2$};
\node (-20) at (0.0,0.8) {$\varepsilon_1$};
\end{tikzpicture}\]
with relations $\varepsilon_i^3=0$ for $i=1,2$ and $\varepsilon_2a_{1}=a_{1}\varepsilon_1$.

\begin{prop}\label{typeG2}
The AR-quiver  $\Gamma_H$ has a good tube of rank $2$ with the mouth modules
$$\xymatrixcolsep{1.5pc}\xymatrixrowsep{1pc}\xymatrix{
1\ar[r]\ar[d]&2\ar[d]&3\ar[l]\\
1\ar[r]\ar[d]&2\ar[d]&\\
1\ar[r]&2&\\
} \ \  \xymatrix{
2\ar[d]&3\ar[l]\\
2\ar[d]&3\ar[l]\\
2.&\\
}$$
\end{prop}

The endomorphism rings of the  two modules are 1-dimensional.

\section{Inhomogeneous  tubes with non-rigid mouth modules and counter examples}\label{sec5}
In this section, we construct inhomogeneous tubes with non-rigid mouth modules. In particular, we  obtain some $\tau$-locally free modules, whose rank vectors are not roots.
Consequently, Conjecture \ref{Conj:gls} fails in general.

Let $\epsilon$ be the following short exact sequence in $\rep(H)$, 
\[
\begin{tikzpicture}[scale=0.5,
 arr/.style={black, -angle 60}]
 \path (2, 1) node (c0) {$0$};
\path (5,1) node (c1) {$X$};
\path (8,1) node (c2) {$Y$};
\path (11,1) node (c3) {$Z$};
\path (14, 1) node (c4) {$0$.};

\draw[->] (c1) edge  (c2);
\draw[->] (c2) edge  (c3);
\draw[->] (c0) edge (c1);
\draw[->] (c3) edge (c4);
\end{tikzpicture}
\]

\begin{lem}\label{lem5.1}
Suppose that $X, ~Y$ and $Z$ are locally free $H$-modules.
\begin{itemize}
\item[(1)] If $\textup{Hom}_H(X,H)=0$, then $\tau \epsilon$ is also a short exact sequence.

\item[(2)] If $\textup{Hom}_H(DH,Z)=0$, then $\tau^{-1} \epsilon$ is also a short exact sequence.
\end{itemize}
\end{lem}

\begin{proof}
Since the projective and injective dimensions of a locally free $H$-module are at most one, we have the following two
exact sequences,
\[
\begin{tikzpicture}[scale=0.6,
 arr/.style={black, -angle 60}]
 \path (2, 1) node (c0) {$0$};
\path (5,1) node (c1) {$\tau X$};
\path (8,1) node (c2) {$\tau Y$};
\path (11,1) node (c3) {$\tau Z$};

\path (14,1) node (c4) {$\nu X$};
\path (17,1) node (c5) {$\nu Y$};
\path (20,1) node (c6) {$\nu Z$};

\path (23, 1) node (c7) {$0$,};

\draw[->] (c0) edge (c1);
\draw[->] (c1) edge (c2);
\draw[->] (c2) edge (c3);
\draw[->] (c3) edge (c4);
\draw[->] (c4) edge (c5);
\draw[->] (c5) edge (c6);
\draw[->] (c6) edge (c7);
\end{tikzpicture}
\]
and
\[
\begin{tikzpicture}[scale=0.6,
 arr/.style={black, -angle 60}]
 \path (2, 1) node (c0) {$0$};
\path (5,1) node (c1) {$\nu^{-1}  X$};
\path (8,1) node (c2) {$\nu^{-1}  Y$};
\path (11,1) node (c3) {$\nu^{-1} Z$};

\path (14,1) node (c4) {$\tau^{-1}  X$};
\path (17,1) node (c5) {$\tau^{-1}  Y$};
\path (20,1) node (c6) {$\tau^{-1} Z$};

\path (23, 1) node (c7) {$0$.};

\draw[->] (c0) edge (c1);
\draw[->] (c1) edge (c2);
\draw[->] (c2) edge (c3);
\draw[->] (c3) edge (c4);
\draw[->] (c4) edge (c5);
\draw[->] (c5) edge (c6);
\draw[->] (c6) edge (c7);
\end{tikzpicture}
\]
Now the result follows since $\nu X=D\text{Hom}_H(X,H)$ and $\nu^{-1}Z=\text{Hom}_H(DH,Z)$.
\end{proof}

\begin{lem}\label{lem5.2}
If the two end terms $X$ and $Z$ in $\epsilon$ are $\tau$-locally free and regular, then so is any indecomposable summand of the middle $Y$.
\end{lem}

\begin{proof} First note that $Y$ is locally free, since it is an extension of locally free modules.
As $X$ and $Z$ are $\tau$-locally free and regular, \[\tau^kX\neq0,~ \tau^kZ\neq0,\]
 and by \cite[Corollary 11.3 ]{[GLS1]},  \[\text{Hom}_H(\tau^kX, H)=\text{Hom}_H(DH, \tau^kZ)=0,\]
 for all $k\in\mathbb{Z}$.
So by Lemma \ref{lem5.1}, the  sequence
\[
\begin{tikzpicture}[scale=0.6,
 arr/.style={black, -angle 60}]
\path (0, 1) node (c-1) {$\tau^k\epsilon$ :};
\path (2, 1) node (c0) {$0$};
\path (5,1) node (c1) {$\tau^k X$};
\path (8,1) node (c2) {$\tau^k Y$};
\path (11,1) node (c3) {$\tau^k Z$};
\path (14, 1) node (c4) {$0$};

\draw[->] (c1) edge  (c2);
\draw[->] (c2) edge  (c3);
\draw[->] (c0) edge (c1);
\draw[->] (c3) edge (c4);
\end{tikzpicture}
\]
is exact. Consequently, by comparing the dimensions, for any indecomposable summand $Y_i$ of $Y$,
$\tau^kY_i\neq 0$ and is locally free for all $k\in\mathbb{Z}$.
Hence $Y_i$ is $\tau$-locally free and regular.
\end{proof}

 \begin{thm} \label{main2}
Let $C$ be a Cartan matrix of type $\type{B}_n$, $\type{CD}_n$, $\type{F}_{41}$ or $\type{G}_{21}$ and $D$ a minimal symmetriser.
\begin{itemize}
\item[(1)] The AR-quiver $\Gamma_H$ has an inhomogeneous tube of $\tau$-locally free modules,
whose mouth modules are not rigid and have $\delta$ as their rank vectors.

\item[(2)] There exist  $\tau$-locally free $H$-modules such that  their rank vectors are not  roots. Consequently, Conjecture 1 fails in these four types.
 \end{itemize}
 \end{thm}

We will prove the theorem  by construction in each case. 

\subsection{Type $\type{B}_n$}
In this case, $D=(1, 2, \dots, 2, 1)$ and  $H$ is as defined in \S \ref{sec:homogBn}, where
$m=1$.
Let $Z$ be the following locally free $H$-module,
$$\xymatrixcolsep{1.5pc}\xymatrixrowsep{1pc}\xymatrix{
&&2\ar[d]\ar[r]&3\ar[r]\ar[d]&\cdots\ar[r]&n-1 \ar[r]\ar[d] & n\ar[d]\\
 & 1\ar[r]& 2\ar[r] & 3\ar[r] & \cdots \ar[r] & n-1\ar[r] & n\ar[r] & n+1.\\
 }$$

\begin{prop} \label{prop:proof1} The module
$Z$ is not rigid, and is a  mouth module in a rank $n$ tube of $\tau$-locally free modules.
Moreover, $\rankv Z=\delta$.
\end{prop}

\begin{proof} Observe that $\End Z\cong K$ and so $Z$ is indecomposable. 
$Z$ has the following projective resolution,
\[
\begin{tikzpicture}[scale=0.7,
 arr/.style={black, -angle 60}]
 \path (-1.5, 1) node (c0) {$0$};
\path (1.5,1) node (c1) {$P_2\oplus P_{n+1}$};
\path (7,1) node (c2) {$P_1\oplus P_{2}$};
\path (11,1) node (c3) {$Z$};
\path (14, 1) node (c4) {$0$,};

\draw[->] (c1) edge node[auto]{$\left(
                  \begin{smallmatrix}
                    r_{a_{1}}& 0\\
                    -r_{\varepsilon_2}& r_{\rho}\\
                  \end{smallmatrix}\right)$} (c2);
\draw[->] (c2) edge node[auto]{}  (c3);
\draw[->] (c0) edge (c1);
\draw[->] (c3) edge (c4);
\end{tikzpicture}
\]
where $\rho=a_{n}\cdots a_{3}a_{2}$.
Computing $\ker  \begin{pmatrix}  l^*_{a_{1}} & 0\\
                    -l^*_{\varepsilon_2} & l^*_{\rho}\\
                  \end{pmatrix}$
gives
$$\xymatrixcolsep{1.5pc}\xymatrixrowsep{1pc}\xymatrix{
&& &3\ar[r]\ar[d]&4\ar[r]\ar[d]&\cdots\ar[r]&n-1 \ar[r]\ar[d] & n\ar[d]\\
\tau Z\colon & 1\ar[r]& 2\ar[r]\ar[d] & 3\ar[r] & 4\ar[r] & \cdots \ar[r] & n-1\ar[r] & n\ar[r] & n+1.\\
&&2&&&\\
 }$$
Repeating the computation, we have
 $$\xymatrixcolsep{1.5pc}\xymatrixrowsep{1pc}\xymatrix{
&& & &4\ar[r]\ar[d]&5\ar[r]\ar[d]&\cdots\ar[r]&n-1 \ar[r]\ar[d] & n\ar[d]\\
\tau^2 Z\colon & 1\ar[r]& 2\ar[r]\ar[d] & 3\ar[r]\ar[d] & 4\ar[r] &5\ar[r] & \cdots \ar[r] & n-1\ar[r] & n\ar[r] & n+1,\\
&&2\ar[r]&3&&\\
 }$$
and so on until
 $$\xymatrixcolsep{1.5pc}\xymatrixrowsep{1pc}\xymatrix{
&& & && && n\ar[d]\\
\tau^{n-2} Z\colon & 1\ar[r]& 2\ar[r]\ar[d] & 3\ar[r]\ar[d] &  \cdots \ar[r]& n-2 \ar[r]\ar[d] & n-1\ar[r]\ar[d] & n\ar[r] & n+1\\
&&2\ar[r]&3\ar[r]&\cdots \ar[r] & n-2\ar[r] & n-1, & &\\
 }$$
$$\xymatrixcolsep{1.5pc}\xymatrixrowsep{1pc}\xymatrix{
\tau^{n-1} Z\colon & 1\ar[r]& 2\ar[r]\ar[d] & 3\ar[r]\ar[d] &  \cdots \ar[r] & n-1\ar[r]\ar[d] & n\ar[r]\ar[d] & n+1&&\\
&&2\ar[r]&3\ar[r]&\cdots \ar[r]  & n-1\ar[r] &n, & &&\\
 }$$
and finally, $\tau^n Z\cong Z.$

Note that any proper locally free submodule of $Z$ is supported at a Dynkin subquiver, although
the support of $Z$ itself is $Q$. Therefore
if $Z$ were not a mouth module, then it would have a regular $\tau$-locally free submodule $X$ with
$\rankv X$ a real Schur root (see Lemma \ref{lem0.1}). However, such regular real Schur roots all appear
 in the inhomogeneous tube constructed in Proposition \ref{typeB}. Moreover, by Theorem \ref{thm:Schurroot},
 they do not appear in any other components of the AR-quiver $\Gamma_H$.
 Therefore, $Z$ must be a mouth module.

 Note that 
$\langle Z,Z\rangle= \text{dim End}Z-\text{dim Ext}^1_H(Z,Z)=0$ since $\rankv Z=\delta$.
Therefore,  $Z$ is not rigid. This completes the proof.
\end{proof}

Let $Y$ be the following extension of $Z$ by the generalised simple module $E_2$,
$$\xymatrixcolsep{1.5pc}\xymatrixrowsep{1pc}\xymatrix{
&&2\ar[r]\ar[d]&3\ar[r]\ar[d]&\cdots\ar[r]&n-1 \ar[r]\ar[d] & n\ar[d]\\
Y= & 1\ar[r]\ar[rd]& 2\ar[r] & 3\ar[r] & \cdots \ar[r] & n-1\ar[r] & n\ar[r] & n+1\\
&&2\ar[d] &&&&\\
&&2 }$$
We have a short exact sequence
$0\longrightarrow E_2\longrightarrow Y\longrightarrow Z\longrightarrow 0$.

\begin{prop} The module $Y$ is $\tau$-locally free, but
\[\rankv Y=\delta+\alpha_2\] is not a root.
Consequently, $\rankv \tau^iY$ is not a root for any 
$i\in \mathbb{Z}$.
\end{prop}

\begin{proof} Note that $\alpha_2$ is a long root, by Proposition \ref{prop:roots},
$\rankv Y=\delta+\alpha_2$ is not a root. Direct computation shows that $\End Y$
is a 3-dimensional local ring and so $Y$ is indecomposable.
On the other hand, we know that both $Z$ and $E_2$ are
$\tau$-locally free and regular, and so by Lemma \ref{lem5.2},  $Y$ is $\tau$-locally free.
\end{proof}

\begin{rem} As in the proof of Proposition \ref{prop:proof1}, by repeatedly computing
$\tau^k Y$, we can also see that $Y$ is $\tau$-locally free and regular. Indeed, we have
  $$\xymatrixcolsep{1.5pc}\xymatrixrowsep{1pc}\xymatrix{
&&&3\ar[r]\ar[d]& 4\ar[r]\ar[d]&\cdots\ar[r]&n-1 \ar[r]\ar[d] & n\ar[d]\\
\tau Y\colon & 1\ar[r]& 2\ar[r]\ar[rd]\ar[d] & 3\ar[r] & 4\ar[r] & \cdots \ar[r] & n-1\ar[r] & n\ar[r] & n+1,\\
&&2\ar[rd] &3\ar[d]&&&\\
&& & 3 }$$
$$\xymatrixcolsep{1.5pc}\xymatrixrowsep{1pc}\xymatrix{
&&&&4\ar[r]\ar[d]& 5\ar[r]\ar[d]&\cdots\ar[r]&n-1 \ar[r]\ar[d] & n\ar[d]\\
\tau^2 Y\colon & 1\ar[r]& 2\ar[r]\ar[d] & 3\ar[r]\ar[rd]\ar[d] & 4\ar[r] & 5\ar[r] &\cdots \ar[r] & n-1\ar[r] & n\ar[r] & n+1,\\
&&2\ar[r] &3\ar[rd]& 4\ar[d] &&\\
&& && 4 }$$
$$\xymatrixcolsep{1.5pc}\xymatrixrowsep{1pc}\xymatrix{\vdots& &\vdots & &\vdots& &\vdots &&\vdots&&\vdots&&\vdots}$$
$$\xymatrixcolsep{1.5pc}\xymatrixrowsep{1pc}\xymatrix{
&&&&&& & n\ar[d]&\\
\tau^{n-2} Y\colon & 1\ar[r]& 2\ar[r]\ar[d] & 3\ar[r]\ar[d] & \cdots\ar[r] & n-2\ar[r]\ar[d] & n-1\ar[r]\ar[d]\ar[rd] &n \ar[r] & n+1,\\
&&2\ar[r] &3\ar[r]& \cdots\ar[r]& n-2\ar[r] &n-1\ar[rd]& n\ar[d]& &\\
&& && &&&n & &\\ \\
}$$
$$\xymatrixcolsep{1.5pc}\xymatrixrowsep{1pc}\xymatrix{
&1\ar[r]& 2\ar[r]\ar[d] & 3\ar[r]\ar[d] & \cdots\ar[r] & n-1\ar[r]\ar[d] & n\ar[d]\ar[r]&n+1\\
\tau^{n-1} Y\colon & 1\ar[r]\ar[rd]& 2\ar[r] & 3\ar[r] & \cdots\ar[r] &  n-1\ar[r] &n \ar[r] & n+1\\
&&2\ar[r] \ar[d] &3\ar[r]\ar[d]& \cdots\ar[r]& n-1\ar[r]\ar[d]& n\ar[d]&&& & \\
&1\ar[r]& 2\ar[r]&3\ar[r]& \cdots\ar[r] &n-1 \ar[r]&n\ar[r] & n+1&&&\\
}$$
and finally, $\tau^n Y\cong Y$.
We remark that each $\tau^k Y$ is an extension of $\tau^k Z$ by $\tau^k E_2$.
\end{rem}

\subsection{Type $\type{CD}_n$}
Use the same symmetriser and orientation as \S \ref{typeCD}.
Consider the following locally free $H$-module,
$$
\xymatrixcolsep{1pc}\xymatrixrowsep{1pc}\xymatrix{
Z:&1\ar[r]&3\ar[r]&4\ar[r]&\cdots\ar[r]&n\ar[r]&n+1\ar[d]\\
&2\ar[r]&3\ar[r]&4\ar[r]&\cdots\ar[r]&n\ar[r]&n+1.\\
}$$

\begin{prop}
The module $Z$ is not rigid, and is a mouth module in a rank 2 tube of $\tau$-locally free modules. Moreover, $\rankv Z=\delta$.
\end{prop}

\begin{proof} Observe that $\End Z\cong K$ and so $Z$ is indecomposable.
 $Z$ has the following projective resolution,
 \[
\begin{tikzpicture}[scale=0.7,
 arr/.style={black, -angle 60}]
 \path (0, 1) node (c0) {$0$};
\path (3,1) node (c1) {$P_{n+1}$};
\path (7,1) node (c2) {$P_1\oplus P_{2}$};
\path (11,1) node (c3) {$Z$};
\path (14, 1) node (c4) {$0$,};

\draw[->] (c1) edge node[auto]{$\left(
                  \begin{smallmatrix}
                    r_{\rho}\\
                    r_{\rho'}\\
                  \end{smallmatrix}\right)$} (c2);
\draw[->] (c2) edge node[auto]{}  (c3);
\draw[->] (c0) edge (c1);
\draw[->] (c3) edge (c4);
\end{tikzpicture}
\]
where $\rho=\varepsilon_{n+1}a_n\dots a_3a_1$  and $\rho'=a_n\dots a_3a_2$.
We have $\tau Z=\ker   \begin{pmatrix} l^*_{\rho}\\  l^*_{\rho'}\\ \end{pmatrix} $ is as follows,
 $$
\xymatrixcolsep{1pc}\xymatrixrowsep{1pc}\xymatrix{
&2\ar[r]&3\ar[r]&4\ar[r]&\cdots\ar[r]&n\ar[r]&n+1\ar[d]\\
&1\ar[r]&3\ar[r]&4\ar[r]&\cdots\ar[r]&n\ar[r]&n+1,\\
}$$
and similarly, $\tau^2Z=Z$. Now by similar arguments as in the proof of Proposition \ref{prop:proof1},
$Z$ is a mouth module in a rank 2 tube of $\tau$-locally free modules.
\end{proof}

Recall  the second mouth module in Proposition \ref{typeCD2},
\[
\xymatrixcolsep{1pc}\xymatrixrowsep{1pc}\xymatrix{
&1\ar[r]&3\ar[r]&4\ar[r]&\cdots\ar[r]&n\ar[r]&n+1\ar[d]\\
&1\ar[r]&3\ar[r]&4\ar[r]&\cdots\ar[r]&n\ar[r]&n+1,\\
}\]
which we denote by $X$.
Consider the following extension $Y$ of $Z$ by $X$,
$$
\xymatrixcolsep{1pc}\xymatrixrowsep{1pc}\xymatrix{
&1\ar[r]&3\ar[r]&4\ar[r]&\cdots\ar[r]&n\ar[r]&n+1\ar[d]\\
&2\ar[r]\ar[rd]&3\ar[r]&4\ar[r]&\cdots\ar[r]&n\ar[r]&n+1\\
&1\ar[r]&3\ar[r]&4\ar[r]&\cdots\ar[r]&n\ar[r]&n+1\ar[d]\\
&1\ar[r]&3\ar[r]&4\ar[r]&\cdots\ar[r]&n\ar[r]&n+1.\\
}$$

\begin{prop}
The module $Y$ is  $\tau$-locally free, but
\[\rankv Y=\delta+\rankv X\] is not a root.
Consequently, $\rankv \tau^iY$ is not a root for any $i\in \mathbb{Z}$.
\end{prop}

\begin{proof}
Note that $\rankv X$ is a long root, so by Proposition \ref{prop:roots}, $\rankv Y$ is not a root.
Observe that  
 $\text{End}Y$ is a 3-dimensional local ring and so $Y$ is indecomposable.
As both $X$ and $Z$ are $\tau$-locally free and regular, by Lemma \ref{lem5.2}, $Y$ is $\tau$-locally free.
\end{proof}

\begin{rem} By computing $\tau^k Y$, we see that $\tau^2 Y\cong Y$. Therefore, $Y$ is contained in
a tube of $\tau$-locally free modules of rank 2 with $\tau Y$  as follows,

$$
\xymatrixcolsep{1pc}\xymatrixrowsep{1pc}\xymatrix{
&2\ar[r]&3\ar[r]&4\ar[r]&\cdots\ar[r]&n\ar[r]&n+1\ar[d]\\
 &1\ar[r]\ar[rd]&3\ar[r]&4\ar[r]&\cdots\ar[r]&n\ar[r]&n+1\\
&2\ar[r]&3\ar[r]&4\ar[r]&\cdots\ar[r]&n\ar[r]&n+1\ar[d]\\
&2\ar[r]&3\ar[r]&4\ar[r]&\cdots\ar[r]&n\ar[r]&n+1.\\
}$$
\end{rem}

\subsection{Type $\type{F}_{41}$}
Use the same symmetriser and orientation as \S \ref{type F1}.
Let $Z$ be the following locally free module,
$$\xymatrixcolsep{1.5pc}\xymatrixrowsep{1pc}\xymatrix{
& 1\ar[r] &2\ar[r] &3\ar[r]& 4\ar[d]&5\ar[l]\ar[d]\ar[lddd]\\
&&&&4&5\ar[l]\\
&&2\ar[r]&3\ar[r]&4\ar[d]&\\
&&&3\ar[r]&4.
}$$
By change of basis, we see that $Z$ is isomorphic to the following module

$$\xymatrixcolsep{1.5pc}\xymatrixrowsep{1pc}\xymatrix{
 & &2\ar[r] &3\ar[r]& 4\ar[d]&\\
Z':&&&3\ar[r]&4&\\
&1\ar[r]&2\ar[r]\ar[ru]&3\ar[r]&4\ar[d]&5\ar[l]\ar[d]\\
&&&&4&5.\ar[l]
}$$

\begin{prop} The module
$Z$ is not rigid, and is a mouth module in a rank 3 tube of $\tau$-locally free modules. Moreover, $\rankv Z=\delta$.
\end{prop}

\begin{proof}
We compute $\tau^k Z$, and obtain the following,
$$\xymatrixcolsep{1.5pc}\xymatrixrowsep{1pc}\xymatrix{
&&3\ar[r] &4\ar[d]&5\ar[d]\ar[l]\\
\tau Z\colon\ \ 1\ar[r]&2\ar[ru]\ar[rd] && 4&5\ar[l]\\
&&3\ar[r]&4\ar[d]&\\
&2\ar[r]&3\ar[r]&4
}\ \ \
\xymatrix{
&1\ar[r]&2\ar[r]&3\ar[r] &4\ar[d]&\\
&\tau^2 Z\colon \ \ &2\ar[r] &3\ar[r]\ar[rd]& 4&\\
&&& &4\ar[d]&5\ar[d]\ar[l]\\
&&&3\ar[r]&4 & 5\ar[l]
}$$
and $\tau^3Z=Z'\cong Z$.

Since the component $(\rankv Z)_1$ of $\rankv Z$ at vertex 1 is $1$,
if $Z$ were not a mouth module, then it would have a $\tau$-locally free submodule
or factor module,  that is not supported at 1.
However, this is not possible,
by similar arguments as in the proof of Proposition \ref{prop:proof1}. Therefore,
$Z$ is a mouth module.
\end{proof}

Recall the first mouth module in Proposition \ref{typeF1} (2), which we denote by $X$,
$$\xymatrixcolsep{1.5pc}\xymatrixrowsep{1pc}\xymatrix{
 &2\ar[r] &3\ar[r]& 4\ar[d]&\\
&2\ar[r]&3\ar[r]&4.\\
}$$
Let $Y$ be the extension of $Z$ by $X$ as follows,
$$\xymatrixcolsep{1.5pc}\xymatrixrowsep{1pc}\xymatrix{
& 1\ar[r] &2\ar[r] &3\ar[r]& 4\ar[d]&5\ar[l]\ar[d]\ar[lddd]\ar[ldddd]\\
&&&&4&5\ar[l]\ar[ldddd]\\
&&2\ar[r]&3\ar[r]&4\ar[d]&\\
&&&3\ar[r]&4\\
&&2\ar[r]&3\ar[r]&4\ar[d]&\\
&&2\ar[r]&3\ar[r]&4\\
}$$
\begin{prop}\label{typeF2}
The module $Y$ is $\tau$-locally free, but $$\rankv Y=\delta+\rankv X $$ is not a  positive root. Consequently, $\rankv \tau^iY$ is not a root for any $i\in \mathbb{Z}$.
\end{prop}

\begin{proof}
As the module structure of $Y$ is slightly more complicated, we  give some more details on its endomorphism ring
to confirm that $Y$ is indecomposable.
 The module $Y$ is a representation as follows:
\[\begin{tikzpicture}[scale=2]
\node (-20) at (0,0) {$k$};
\node (-20) at (1.5,0) {$k^4$};
\node (-20) at (3,0) {$k^5$};
\node (-20) at (4.5,0) {$k^6$};
\node (-20) at (6,0) {$k^2$};

\node (-20) at (0.75,0.35) {$\left(
                              \begin{smallmatrix}
                                1 \\
                                0 \\
                                0 \\
                                0 \\
                              \end{smallmatrix}
                            \right)
$};
\node (-20) at (2.25,0.4) {$\left(
                              \begin{smallmatrix}
                                1&0&0&0\\
                                0&1&0&0 \\
                                0&0&0&0 \\
                                0&0&1&0 \\
                                0&0&0&1 \\
                              \end{smallmatrix}
                            \right)
$};

\node (-20) at (3.75,0.4) {$\left(
                              \begin{smallmatrix}
                                1&0&0&0&0 \\
                                0&0&0&0&0 \\
                                0&1&0&0&0 \\
                                0&0&1&0&0 \\
                                0&0&0&1&0  \\
                                0&0&0&0&1 \\
                              \end{smallmatrix}
                            \right)
$};

\node (-20) at (4.5,0.5){$f$};

\node (-20) at (5.35,0.4) {$\left(
                              \begin{smallmatrix}
                                1&0\\
                                0&1 \\
                                0&0 \\
                                1&0 \\
                                1&0\\
                                0&1 \\
                              \end{smallmatrix}
                            \right)
$};

\draw [->] (0.25,0) -- (1.25,0);
\draw [->] (1.75, 0) -- (2.75,0);
\draw [->] (3.25, 0) -- (4.25,0);
\draw [->] (5.75, 0) -- (4.75,0);

\path (6.1, 0.45) node {$g$};

\draw[-latex] (5.9, 0.1) .. controls (5.7, 0.4) and (6.3,0.4) .. (6.1, 0.05);
\draw[-latex] (4.4, 0.1) .. controls (4.2,0.4) and (4.8,0.4) .. (4.6,0.05);
\end{tikzpicture}\]
where $f$ is a  block-wise $6\times 6$ matrix with the diagonal blocks
$\begin{pmatrix} 0&0\\ 1&0\end{pmatrix}$ and the others zero, and
$g=\begin{pmatrix} 0&0\\ 1&0\end{pmatrix}$.
Let $\varphi=(\varphi_1,\varphi_2,\varphi_3,\varphi_4,\varphi_5)\in \text{End}Y$,
computing the commutative relations defining a homomorphism of representations gives
$$\varphi=\left(a,\left(
                              \begin{smallmatrix}
                                a&0&0&0 \\
                                0&a&0&0 \\
                                0&b&a&0 \\
                                0&c&-b&a  \\
                              \end{smallmatrix}
                            \right), \left(
                              \begin{smallmatrix}
                                a&0&0&0&0 \\
                                0&a&0&0&0 \\
                                0&0&a&0&0 \\
                                0&b&0&a&0 \\
                                0&c&b&-b&a  \\
                              \end{smallmatrix}
                            \right), \left(
                              \begin{smallmatrix}
                                a&0&0&0&0&0 \\
                                0&a&0&0&0&0 \\
                                0&0&a&0&0&0 \\
                                0&0&0&a&0&0 \\
                                0&0&b&0&a&0 \\
                                0&0&c&b&-b&a  \\
                              \end{smallmatrix}
                            \right),\left(
                              \begin{smallmatrix}
                                a&0 \\
                                0&a \\
                              \end{smallmatrix}
                            \right)\right).$$
Therefore, $\text{End}Y$ is a local ring, and so  $Y$ is indecomposable.

Next, by Lemma \ref{lem5.2}, $Y$ is $\tau$-locally free, since it is an extension of two regular $\tau$-locally free modules.
Finally,  $\rankv Y=\delta+\rankv X$ is not a root, since $\rankv X$ is a long root. So $Y$ is a module as claimed.
\end{proof}

\subsection{Type $\type{G}_{21}$} Use the same symmetriser and orientation as \S \ref{type g1}.
Let $Z$ be the following locally free $H$-module
$$\xymatrixcolsep{1.5pc}\xymatrixrowsep{1pc}\xymatrix{
&&2\ar[r]&3\ar[d]\\
&1\ar[r]&2\ar[r]&3\ar[d]\\
&&&3.\\
}$$

\begin{prop}
The module $Z$ is not rigid, and is a mouth module in a rank 2 tube of $\tau$-locally free modules. Moreover, $\rankv Z=\delta$.
\end{prop}

\begin{proof}
By computing $\tau^kZ$, we have $\tau^2Z=Z$ and $\tau Z$ is as follows,
$$\xymatrixcolsep{1.5pc}\xymatrixrowsep{1pc}\xymatrix{
&1\ar[r] & 2\ar[r] &3\ar[d]\\
& &&3\ar[d]\\
&&2\ar[r]&3.}$$

By similar arguments as in the proof of Proposition \ref{prop:proof1},
$Z$ is a mouth module.
\end{proof}

Recall the second mouth module from Proposition \ref{typeG1}, which we denote by $X$,
 $$\xymatrixcolsep{1.5pc}\xymatrixrowsep{1pc}\xymatrix{
&2\ar[r]&3\ar[d]\\
&2\ar[r]&3\ar[d]\\
&2\ar[r]&3.\\
}$$
Let $Y$ be the following extension of $Z$ by $X$,
$$\xymatrixcolsep{1.5pc}\xymatrixrowsep{1pc}\xymatrix{
&&2\ar[r]&3\ar[d]\\
&1\ar[r]\ar[rdd]&2\ar[r]&3\ar[d]&\\
&&&3\\
&&2\ar[r]&3\ar[d]\\
&&2\ar[r]&3\ar[d]\\
&&2\ar[r]&3.\\
}$$

\begin{prop}
The module $Y$ is $\tau$-locally free, but $$\rankv Y=\delta+\rankv X$$ is not a root. Consequently, $\rankv \tau^iY$ is not a root for any $i\in \mathbb{Z}$.
\end{prop}

\begin{proof}
As in the proof of Proposition \ref{typeF22}, we compute endomorphisms $\varphi=(\varphi_1,\varphi_2,\varphi_3)$
of $Y$. The module $Y$ is a representation as follows,
\[\begin{tikzpicture}[scale=2]
\node (-20) at (0,0) {$k$};
\node (-20) at (1.5,0) {$k^5$};
\node (-20) at (3,0) {$k^6$,};

\node (-20) at (0.75,0.35) {$\left(
                              \begin{smallmatrix}
                                0 \\
                                1 \\
                                1 \\
                                0 \\
                                0 \\
                              \end{smallmatrix}
                            \right)$};
\node (-20) at (2.25,0.4) {$\left(
                              \begin{smallmatrix}
                                1&0&0&0&0 \\
                                0&1&0&0&0 \\
                                0&0&0&0&0 \\
                                0&0&1&0&0 \\
                                0&0&0&1&0  \\
                                0&0&0&0&1  \\
                              \end{smallmatrix}
                            \right)$};
\node (-20) at (3,0.45) {$f$};
\draw [->] (0.25,0) -- (1.25,0);
\draw [->] (1.75, 0) -- (2.75,0);
\draw[-latex] (2.9,0.1) .. controls (2.7,0.4) and (3.3,0.4) .. (3.1,0.05);
\end{tikzpicture}\]
where $f$ is the blockwise $6\times 6$ matrix with 2 diagonal blocks of the form
$\left(\begin{smallmatrix} 0&0&0\\  1&0&0\\  0&1&0\\\end{smallmatrix}\right)$
and the others zero.
%
Computing the commutative  relations
defining a homomorphism of representations gives,
$$\varphi=\left(a,\left(
                              \begin{smallmatrix}
                                a&0&0&0&0 \\
                                0&a&0&0&0 \\
                                b&0&a&0&0 \\
                                c&b&-b&a&0 \\
                                d&c&-c&-b&a  \\
                              \end{smallmatrix}
                            \right), \left(
                              \begin{smallmatrix}
                                a&0&0&0&0&0 \\
                                0&a&0&0&0&0 \\
                                0&0&a&0&0&0 \\
                                b&0&0&a&0&0 \\
                                c&b&0&-b&a&0 \\
                                d&c&b&-c&-b&a  \\
                              \end{smallmatrix}
                            \right)\right).$$
Therefore, $\text{End}Y$ is a local ring, and so  $Y$ is indecomposable.
Now by similar arguments as in Proposition \ref{prop:proof1}, we conclude that $Y$ is
a module as claimed.
\end{proof}

\subsection{A remark}
\begin{itemize}\item[(1)]
By computing $\tau^k Y$, we know that $Y$ is $\tau$-periodic. So any $Y$ constructed in this section is contained in a  tube. The tubes that contain those $Y$ do not have rigid 
mouth modules,
since rigid $\tau$-locally free modules are uniquely determined by their rank vectors and appear
in the inhomogeneous tubes discussed in the proofs of Proposition \ref{Prop:typeA} and Theorem  \ref{thm3.1}.

\item[(2)] Following Proposition \ref{cor}, classifying $\tau$-locally free $H$-modules when $C$ is affine amounts to classifying components of
the form $\mathbb{Z}\mathbb{A}_{\infty}$ and tubes with non-rigid mouth modules. We believe that more such tubes exist, but we don't have any example of components of the form $\mathbb{Z}\mathbb{A}_{\infty}$.
\end{itemize}

\vspace{2mm}\noindent {\bf Acknowledgements}  The  authors would like to thank Professor Shiping Liu for
the helpful discussions on regular components of AR-quivers.

\end{document}